\documentclass[11pt]{amsart}   %%%%%%%Please do not change
\usepackage{amscd}                  %%%%%%%
\input{xy}
\xyoption{all}

\newtheorem{theorem}{Theorem}[section]
\newtheorem{lemma}[theorem]{Lemma}
\newtheorem{proposition}[theorem]{Proposition}
\newtheorem{corollary}[theorem]{Corollary}
\newtheorem{problem}[theorem]{Problem}
\newtheorem{question}[theorem]{Question}

\theoremstyle{definition}
\newtheorem{definition}[theorem]{Definition}
\newtheorem{remark}[theorem]{Remark}
\numberwithin{equation}{section}

\newcommand{\w}{\omega}
\newcommand{\LC[1]}{\mathrm{LC}^{#1}}
\newcommand{\lc[1]}{{lc}^{#1}}
\newcommand{\U}{\mathcal U}
\newcommand{\C}{\mathcal C}
\newcommand{\V}{\mathcal V}
\newcommand{\K}{\mathcal K}
\newcommand{\pr}{\operatorname{pr}}
\newcommand{\e}{\varepsilon}
\newcommand{\Ra}{\Rightarrow}
\newcommand{\IZ}{\mathbb Z}
\newcommand{\IN}{\mathbb N}
\newcommand{\IQ}{\mathbb Q}

\newcommand{\IR}{\mathbb R}

\newcommand{\simU}{\underset{\U}{\sim}}
\newcommand{\diam}{\mathrm{diam}}
\newcommand{\dist}{\mathrm{dist}}
\newcommand{\edim}{\mbox{e-dim}}
\newcommand{\trind}{\mathrm{trind}\,}

\newcommand{\trt}{\mathsf{trt}}

\newcommand{\Z}{\mathcal Z}
\newcommand{\DZ}{\overline{\mathcal Z}}
\newcommand{\Tor}{\mathrm{Tor}}
\newcommand{\pTor}{p\mbox{-}\mathrm{Tor}}
\newcommand{\Top}{\mathsf{Top}}
\newcommand{\Baire}{\mathsf{Br}}
\newcommand{\Rp}{R_p}
\newcommand{\zetaa}{z}

\begin{document}

\title[On homotopical and homological $Z_n$-sets]%
{On homotopical and homological $Z_n$-sets}

%    Information for first author:
\author{Taras Banakh}
\address{Ivan Franko National University of Lviv (Ukraine),\newline
\indent Uniwersytet Humanistyczno-Przyrodniczy, (Kielce, Poland),\newline  \indent Nipissing University, (North Bay, Canada)
}
\email{tbanakh@yahoo.com; t.o.banakh@gmail.com}
\thanks{The substantial part of this paper was written during the stay of the first author in Nipissing University (North Bay, Canada)}

\author{Robert Cauty}
\address{Universit\'e Paris VI (France)}
\email{cauty@math.jussieu.fr}
\author{Alex Karassev}
\address{Nipissing University (North Bay, Canada)}
\email{alexandk@nipissingu.ca}

%    Current address (if needed): 
%\curraddr{}

%    Information for second author (if needed): 
%\author{Author Two}
%\address{}
%\email{}

%    General info
%%%%%%%%%%%%%%%%%%%%%%%%%%%%%%%%%%%%%%%%%%%%%%%%%%%

                                  %
%                                                                                                                           %
%         Please use the current 2010 Mathematics Subject Classification:             %
%         http://www.ams.org/mathscinet/msc/                                                        %
%         http://www.zentralblatt-math.org/msc/en/                                                 %
%%%%%%%%%%%%%%%%%%%%%%%%%%%%%%%%%%%%%%%%%%%%%%%%%%%

\begin{abstract}
A closed subset $A\subset X$ is called a {\em homological} ({\em
homotopical\/}) {\em $Z_n$-set} if for any $k<n+1$ and any open
set $U\subset X$ the relative homology (homotopy) group
$H_k(U,U\setminus A)$ (~$\pi_k(U,U\setminus A)$~) vanishes. A
closed subset $A$ of an $\LC[n]$-space $X$ is a homotopical
$Z_n$-set if and only if $A$ is a $Z_n$-set in $X$ in the sense
that each map $f:[0,1]^n\to X$ can be uniformly approximated by
maps into $X\setminus A$. Applying Hurewicz Isomorphism Theorem we
prove that a homotopical $Z_2$-subset of an $\LC[1]$-space is a homotopical
$Z_n$-set iff it is a homological $Z_n$-set in $X$. From the K\"unneth Formula we derive Multiplication, Division and $k$-Root Formulas for homological $Z_n$-sets.
We prove that the set $\Z_n^{\IZ}(X)$ of homological $Z_n$-points of a metrizable separable $\lc[n]$-space $X$ is of type $G_\delta$ in $X$.
We introduce and study the classes $\Z^{\IZ}_n$ (resp. $\DZ^{\IZ}_n$) of topological spaces $X$ with $\Z^{\IZ}_n(X)=X$ (resp. $\overline{\Z^{\IZ}_n(X)}=X$) and prove Multiplication, Division and $k$-Root Formulas for such classes. We also show that a (locally compact $\lc[n]$-)space $X\in\Z^{\IZ}_n$ has Steinke dimension $t(X)\ge n+1$ (has cohomological dimension $\dim_G(X)\ge n+1$ for any coefficient group $G$). A locally compact ANR-space $X\in\Z_\infty^{\IZ}$ is not a $C$-space and has extension dimension $\edim X\not\le L$ for any non-contractible CW-complex $L$.
\end{abstract}

\maketitle

In this paper we focus on applications of homological methods to
studying $Z_n$-sets in topological spaces. Being higher-dimensional counterparts of closed nowhere dense subsets, $Z_n$-sets are of
crucial importance in infinite-dimensional and geometric
topologies \cite{BRZ}, \cite{BP}, \cite{Dav}, \cite{DW}, \cite{vM}, \cite{BV}
and play a role also in Dimension Theory \cite{Ba} and Theory of Selections
\cite{Us}. $Z_n$-Sets were introduced by H.Toru\'nczyk in
\cite{To}. He defined a closed subset $A$ of a topological space
$X$ to be a {\em $Z_n$-set\/} if any map $f:I^n\to X$ from the
$n$-dimensional cube $I^n=[0,1]^n$ can be uniformly approximated
by maps into $X\setminus A$. Actually, $Z_n$-sets work properly
only in $\LC[n]$-spaces where they coincide with so-called
homotopical $Z_n$-sets. By definition, a closed subset $A$ of a
topological space $X$ is a {\em homotopical $Z_n$-set\/} if for
any open cover $\U$ of $X$ every map $f:I^n\to X$ can be
approximated by a map $f':I^n\to X\setminus A$, $\U$-homotopic to
$f$. In fact, homotopical $Z_n$-sets are nothing else but closed
locally $n$-negligible sets in the sense of H.Toru\'nczyk
\cite{To}.

In Section~\ref{s:Zsets} we apply the Hurewicz Isomorphism Theorem
to characterize homotopical $Z_n$-sets $A$ in Tychonov
$\LC[1]$-spaces $X$ as homotopical $Z_{\min\{n,2\}}$-sets such
that the relative homology groups \mbox{$H_k(U,U\setminus A)$} vanish
for all $k<n+1$ and all open sets $U\subset X$. Having in mind this
characterization of homotopical $Z_n$-sets, we define a closed
subset $A$ of a topological space $X$ to be a {\em homological
$Z_n$-set} (more generally, a {\em $G$-homological $Z_n$-set} for a coefficient group $G$) if $H_k(U,U\setminus A)=0$ (resp. $H_k(U,U\setminus A;G)=0$) for all $k<n+1$ and all
open sets $U\subset X$. Therefore, a homotopical $Z_2$-set in a
Tychonov $\LC[1]$-space $X$ is a homotopical $Z_n$-set if and only
if it is a homological $Z_n$-set. It should be mentioned that
under some restrictions on the space $X$ this characterization of
$Z_n$-sets has been exploited in mathematical literature
\cite[4.2]{DW}, \cite{Dob}, \cite{Kro}. The homological
characterization of $Z_n$-sets makes possible to apply powerful
tools of Algebraic Topology for studying $Z_n$-sets. Homological $Z_n$-sets behave like usual $Z_n$-sets: the union of two $G$-homological $Z_n$-sets is a $G$-homological $Z_n$-set and so is each closed subset of a $G$-homological $Z_n$-set.

In Section~\ref{irred}, applying the technique of irreducible homological barriers, we prove that a closed subset $A\subset
X$ is a homological $Z_n$-set in $X$ if each point $a\in A$ is a
homological $Z_n$-point in $X$ and each closed subset $B\subset A$ with $|B|>1$ can be separated by a homological $Z_{n+1}$-set. This characterization makes it possible to apply Steinke's separation dimension $\mathsf{t}(\cdot)$ and its transfinite extension $\mathsf{trt}(\cdot)$ to studying homological $Z_n$-sets. In particular, we prove that a closed subset $A\subset X$ with finite separation dimension $d=\mathsf{t}(A)$ is a $G$-homological $Z_n$-set in $X$ if each point $a\in A$ is a $G$-homological $Z_{n+d}$-point in $X$. An infinite version of this result asserts that a closed subset $A\subset X$ having transfinite separation dimension $\mathsf{trt}(A)$ is a $G$-homological $Z_\infty$-set in $X$ if and only if each point $a\in A$ is a $G$-homological $Z_\infty$-point in $X$.

In Section~\ref{bock} we develop the Bockstein theory for $G$-homological $Z_n$-sets. The main result is Theorem~\ref{bockstein} asserting that
a subset $A\subset X$ is a $G$-homological $Z_n$-set in $X$ if and only if $A$ is an $H$-homological $Z_n$-set for all groups $H\in\sigma(G)$, where $\sigma(G)\subset\{\IQ,\IZ_p,\IQ_p,\Rp:p$ is prime$\}$ is the Bockstein family of $G$.

The main result of Section~\ref{products} is Multiplication Theorem~\ref{multZn}:  for a homological
(homotopical) $Z_n$-set $A$ in a space $X$ and a homological
(homotopical) $Z_m$-set in $Y$ the product $A\times B$ is a
homological (homotopical) $Z_{n+m+1}$-set in $X\times Y$. For ANRs
the homotopical version of this result has been proved by T.Banakh
and Kh.Trushchak in \cite{BT}.
It is interesting that the Multiplicative Theorem for homological $Z_n$-sets can be partly reversed, which leads to Division and $k$-Root Theorems proved in Section~\ref{div}.

In Section~\ref{Zpoints} we apply the obtained results on
$Z_n$-sets to study $Z_n$-points. We show that the set of
$Z_n$-points in a metrizable separable space $X$ always is a
$G_\delta$-set. Moreover, if $X$ is an $\LC[n]$-space, then the
sets of homological and homotopical points also are $G_\delta$ in $X$.

In Sections~\ref{Zn} and \ref{DZn} we introduce and study the classes $\Z_n$, $\Z_n^{\IZ}$ of spaces whose all points are homotopical (resp. homological) $Z_n$-points, and classes $\DZ_n$ (resp. $\DZ_n^{\IZ}$) of spaces containing dense sets of homotopical (resp. homological) $Z_n$-points. Applying the results from the preceding sections, we prove Multiplication, Division and $k$-Root Formulas for these classes.

In Sections~\ref{sect:dim} and \ref{dimension} we study dimension
properties of spaces from the class $\Z^{\IZ}_n$ (i.e., spaces
whose all points are homological $Z_n$-points). We show that for
any space $X\in\Z_n^{\IZ}$ the transfinite separation dimension
$\mathsf{trt}(X)\ge 1+n$. If, in addition, $X$ is a locally
compact $\LC[n]$-space, then $X$ has cohomological dimension
$\dim_G(X)\ge n+1$ for any group $G$. Also the inequality
$\edim(X)\le L$ for a CW-complex $L$ implies that $\pi_k(L)=0$
for all $k\le n$. Each locally compact locally contractible space
$X\in\Z^{\IZ}_\infty$ is infinite-dimensional in a rather strong
sense: $X$ fails to be a $C$-space and has extension dimension
$\edim X\not\le L$ for any non-contractible CW-complex $L$, see
Theorem~\ref{infdim}.

\section{Preliminaries}

All topological spaces considered in this paper are Tychonov; $I$
stands for the closed interval $[0,1]$, $n$ will denote a
non-negative integer or infinity. In the paper we use singular
homology $H_*(X;G)$ with coefficients in a non-trivial abelian
group $G$. If $G=\IZ$ is the group of integers, we omit the symbol
$\IZ$ and write $H_*(X)$ instead of $H_*(X;\IZ)$. By $\widetilde
H_*(X;G)$ we  denote the singular homology of $X$, reduced
in dimension zero.

Let $\U$ be a cover of a space $X$. Two maps $f,g:Z\to X$ are called
\begin{itemize}
\item {\em $\U$-near} (denoted by $(f,g)\prec\U$) if for any $z\in Z$ there is $U\in\U$ with $\{f(z),g(z)\}\subset U$;
\item {\em $\U$-homotopic} (denoted by $f\simU g$) if there is a homotopy $h:Z\times [0,1]\to X$ such that $h(z,0)=f(z)$, $h(z,1)=g(z)$ and $h(\{z\}\times[0,1])\subset U\in\U$ for all $z\in Z$.
\end{itemize}

There is also a (pseudo)metric counterpart of these notions. Let $\rho$ be a continuous pseudometric on a space $X$ and $\e>0$ be a real number. Two maps $f,g:Z\to X$ are called
\begin{itemize}
\item {\em $\e$-near} if $\dist(f,g)<\e$ where $\dist(f,g)=\sup_{z\in Z}\rho(f(z),g(z))$;
\item {\em $\e$-homotopic} if there is a homotopy $h:Z\times [0,1]\to X$ such that $h(z,0)=f(z)$, $h(z,1)=g(z)$ and $\diam_\rho h(\{z\}\times[0,1])<\e$ for all $z\in Z$.
\end{itemize}

%An open cover $\U$ of a space $X$ is called a {\em uniform cover} if for some continuous pseudometric the cover $\{B_\rho(x,1):x\in X\}$ by 1-balls is inscribed into $X$. By the fundamental Tukey result \cite{??}, each open cover of a paracompact space is uniform.

The following easy lemma helps to reduce the ``cover'' version of (homotopical) nearness to the ``pseudometric''  one.

\begin{lemma}\label{metric} For any open cover $\U$ of a Tychonov space $X$ and any compact set $K\subset X$ there is a continuous pseudometric $\rho$ on $X$ such that each 1-ball $B(x,1)=\{x'\in X:\rho(x,x')<1\}$ centered at a point $x\in K$ lies in some set $U\in\U$.
\end{lemma}

\begin{proof} Embed the Tychonov space $X$ into a Tychonov cube $I^\kappa$
for a suitable cardinal $\kappa$. For each $x\in K$ find a finite
index set $F(x)\subset\kappa$ and an open subset $W_x\subset I^{F(x)}$ whose
preimage $V_x=\pr^{-1}_{F(x)}(W_x)$ under the projection
$\pr_{F(x)}:X\to I^{F(x)}$ contains the point $x$ and lies in some
$U\in\U$. By the compactness of $K$, the open cover $\{V_x:x\in
K\}$ of $K$ contains a finite subcover
$\{V_{x_1},\dots,V_{x_m}\}$. Now consider the finite set
$F=\bigcup_{i=1}^m F(x_i)$ and note that each set $V_{x_i}$ is the
preimage of the some open set $W_i\subset I^F$ under the projection
$\pr_F:X\to I^F$. Let $d$ be any metric on the finite-dimensional
cube $I^F$. By the compactness of $C=\pr_F(K)\subset
\bigcup_{i=1}^m W_i$ there is $\e>0$ such that each $\e$-ball
centered at a point $z\in C$ lies in some $W_i$. Finally, define
the pseudometric $\rho$ on $X$ letting $\rho(x,x')=\frac1{\e}\cdot
d(\pr_F(x),\pr_F(x'))$ for $x,x'\in X$. It is easy to see that
each 1-ball centered at any point $x\in K$ lies in some $U\in\U$.
\end{proof}

Let us recall that a space $X$ is called an {\em $\LC[n]$-space} if for each point $x\in X$, each neighborhood $U$ of $x$, and each $k<n+1$ there is a neighborhood $V\subset U$ of $x$ such that each map $f:\partial I^k\to V$ from the boundary of the $k$-dimensional cube extends to a map $f:I^k\to U$.

The following homotopy approximation theorem for $\LC[n]$-spaces
is well-known, see \cite[p.159]{Hu} or \cite{BV}.

\begin{lemma}\label{LCnear} For any open cover $\U$ of a paracompact $\LC[n]$-space $X$ and any $k< n+1$ there is an open cover $\V$ of $X$ such that any two $\V$-near maps $f,g:K\to X$ from a simplicial complex $K$ of dimension $\dim K\le k$ are $\U$-homotopic.
\end{lemma}

This lemma has a homological counterpart. A space $X$ is defined
to be an {\em $\lc[n]$-space} if for each point $x\in X$, each
neighborhood $U$ of $x$, and each $k<n+1$ there is a neighborhood
$V\subset U$ of $x$ such that the homomorphism $i_*:\widetilde H_k(V)\to
\widetilde H_k(U)$ of singular homologies induced by the inclusion map
$i:V\to U$ is trivial.

It is known that each $\LC[n]$-space is an $\lc[n]$-space, see
\cite{Un}. The converse is true for $\LC[1]$-spaces, see
\cite{Un}. The proof of the following lemma can be found in \cite[5.4]{Bow}

\begin{lemma}\label{lcn} For any paracompact $\lc[n]$-space $X$ and any $k<n+1$ there is an open cover $\U$ of $X$ such that any two $\U$-near maps $f,g:K\to X$ defined on a simplicial complex $K$ of dimension $\dim K\le k$ induce the same homomorphisms $f_{*},g_{*}:H_*(K)\to H_*(X)$ on homologies.
\end{lemma}

In the sequel we shall need another three homological properties of $\lc[n]$-spaces. First, we recall one well-known fact from homological algebra, see Lemma 16.3 in \cite{Bre}.

If the commutative diagram in the category of abelian groups
$$\begin{CD}
 @. A_2 @>>> A_3\\
@. @V{i_2}VV @ V{i_3}VV\\
B_1 @>>> B_2 @>>> B_3\\
@V{i_1}VV @V{i_4}VV\\
C_1 @>>> C_2
\end{CD}
$$ has exact middle row, then the finite generacy (or triviality) of the groups $i_1(B_1)$ and $i_3(A_3)$ implies the finite generacy (triviality) of the group $i_4\circ i_2(A_2)$.

A subset $A$ of a topological space $X$ is called {\em precompact} if $A$ has compact closure in $X$.

\begin{lemma}\label{lcn-compact} Any precompact set $C$ in a Tychonov $\lc[n]$-space $X$ with $n<\infty$ has an open neighborhood $U$ such that the inclusion homomorphisms $H_k(U)\to H_k(X)$ have finitely-generated image for all $k\le n$.
\end{lemma}

\begin{proof} We shall prove this lemma by induction on $n$. For $n=0$ the assertion of the lemma follows from the local connectedness of $\lc[0]$-spaces. Assume that for some $n$ the lemma has been proved for all $k<n$. Consider the family $\mathcal K_n$ of compact subsets $K$ of $X$ having an open neighborhood $U$ such that the inclusion homomorphism $H_n(U)\to H_n(X)$ has finitely generated range. We claim that for any compact subsets $A,B\in\mathcal K_n$ the union $A\cup B\in\mathcal K_n$.
To show this, fix open neighborhoods $O_A,O_B$ of the compacta $A,B$ such that the inclusion homomorphisms $H_n(O_A)\to H_n(X)$ and $H_n(O_B)\to H_n(X)$ have finitely generated ranges. By the inductive assumption, the compact set $A\cap B$ has a neighborhood $O_{A\cap B}\subset O_A\cap O_B$ such that the inclusion homomorphism $H_{n-1}(O_{A\cap B})\to H_{n-1}(O_A\cap O_B)$ has finitely generated range.
Using the complete regularity of $X$, find two open neighborhoods $U_A\subset O_A$ and $U_B\subset O_B$ of the compacta $A,B$ such that $U_A\cap U_B\subset O_{A\cap B}$.
The proof will be completed as soon as we check that the inclusion homomorphism $H_n(U_A\cup U_B)\to H_n(X)$ has finitely-generated range. This follows from the diagram whose rows are Mayer-Vietoris sequences  of the pairs $\{U_A,U_B\}$, $\{O_A,O_B\}$ and $\{X,X\}$:
$$
\xymatrix{
H_n(U_A)\oplus H_n(U_B)\ar[d]\ar[r]& H_n(U_A\cup U_B)\ar[d]^{i_2}\ar[r]& H_{n-1}(U_A\cap U_B)\ar[d]^{i_3}\\
H_n(O_A)\oplus H_n(O_B)\ar[r]\ar[d]_{i_1} &H_n(O_A\cup O_B)\ar[r]\ar[d]^{i_4} &H_{n-1}(O_A\cap O_B)\ar[d]\\
H_n(X)\oplus H_n(X)\ar[r]& H_n(X\cup X)\ar[r]& H_{n-1}(X\cap X)
}
$$
Since the homomorphisms $i_1$ and $i_3$ have finitely generated ranges, so does the homomorphism $i_4\circ i_2:H_n(U_A\cup U_B)\to H_n(X)$.
This proves the additivity of the family $\mathcal K_n$.

To finish the proof of the lemma it remains to check that each compact subset $C$ of $X$ belongs to the family $\mathcal K_n$. By the $\lc[n]$-property of $X$ each point $x\in X$ has an open neighborhood $O_x\subset X$ such that the inclusion homomorphism $H_n(O_x)\to H_n(X)$ has finitely generated range. Take any closed neighborhood $C_x\subset O_x$ of $x$ in the compact subset $C$. It is clear that $C_x\in \mathcal K_n$ for all $x\in C$. By the compactness of $C$ the cover $\{C_x:x\in C\}$ has finite subcover $C_{x_1},\dots,C_{x_m}$. Now the additivity of the family $\mathcal K_n$ implies that $C=\bigcup_{i\le m}C_{x_i}\in\mathcal K_n$.
\end{proof}

\begin{lemma}\label{fingen} Let $X$ be a locally compact $\lc[n]$-space and $V\subset U$ be open subsets of $X$ such that $\overline{V}\subset U$ and $\overline{U}$ is compact. Then for any $k\le n$ the inclusion homomorphism $H_k(X,X\setminus \bar U)\to H_k(X,X\setminus \bar V)$ has finitely generated range.
\end{lemma}

\begin{proof}  Take open sets $W_1\subset W_2\subset W_3\subset X$ such that $W_3$ has compact closure in $X$ and $\bar U\subset W_1\subset\overline{W}_1\subset W_2\subset\overline{W}_2\subset W_3$. The excision property for singular homology (see Theorem 2.20 of \cite{Hat}) implies that the inclusion homomorphisms $H_k(W_1,W_1\setminus \bar U)\to H_k(X,X\setminus \bar U)$ and $H_k(W_3,W_3\setminus \bar V)\to H_k(X,X\setminus \bar V)$ are isomorphisms. Thus it suffices to check that the inclusion homomorphism $H_k(W_1,W_1\setminus\bar U)\to H_k(W_3,W_3\setminus\bar V)$ has finitely generated range. This will be done with help of the commutative diagram whose rows are the exact sequences of the pairs $(W_1,W_1\setminus \bar U)$, $(W_2,W_2\setminus\bar V)$, $(W_3,W_3\setminus\bar V)$ and columns are inclusion homomorphisms in homologies:
%\vskip-20pt
$$
\begin{CD}
@. H_k(W_1,W_1\setminus \bar U)@>{}>> H_{k-1}(W_1\setminus \bar U)\\
@. @ V{i_2}VV @V{i_3}VV\\
H_k(W_2)@>>> H_k(W_2,W_2\setminus \bar V)@>>> H_{k-1}(W_2\setminus \bar V)\\
@ V{i_1}VV @V{i_4}VV\\
H_k(W_3)@>>> H_k(W_3,W_3\setminus\bar V)
\end{CD}
$$
Lemma~\ref{lcn-compact} implies that the homomorphisms  $i_1$  and $i_3$ have finitely generated ranges (because $W_2$ and $W_1\setminus\bar U$ have compact closures in $W_3$ and $W_2\setminus \bar V$, respectively). Consequently, the inclusion homomorphism
$i_4\circ i_2:H_k(W_1,W_1\setminus\bar U)\to H_k(W_3,W_3\setminus\bar V)$ also has finitely generated range.
\end{proof}

\begin{lemma}\label{homtriv} Let $X$ be a locally compact $\lc[n]$-space and $x$ be a point with $H_k(X,X\setminus\{x\})=0$ for some $k\le n$. Then for any neighborhood $U\subset X$ of $x$  there is a neighborhood $V\subset U$ of $x$ such that the inclusion homomorphism $H_k(X,X\setminus \bar U)\to H_k(X,X\setminus\bar V)$ is trivial.
\end{lemma}

\begin{proof} Without loss of generality the neighborhood $U$ has compact closure in $X$. By Lemma~\ref{fingen} there is a neighborhood $W\subset U$ of $x$ such that  the inclusion homomorphism  $i_{UW}:H_k(X,X\setminus \bar U)\to H_k(X,X\setminus\bar W)$ has finitely generated range $\mathrm{Im}(i_{UW})$. Pick up finitely many generators $g_1,\dots,g_m$ of the group $\mathrm{Im}(i_{UW})$. The triviality of the homology group $H_k(X,X\setminus\{x\})$ implies the triviality of  the inclusion homomorphism $j:H_k(X,X\setminus \bar W)\to H_k(X,X\setminus\{x\})$. Then for every $i\le m$ the image $j(g_i)=0$ and we can find a neighborhood $V_{i}\subset W$ of $x$ such that the image of $g_i$ under the inclusion homomorphism $H_k(X,X\setminus\overline{W})\to H_k(X,X\setminus \bar V_{i})$ is trivial. For the neighborhood $V=\bigcap_{i\le m}V_{i}$ of $x$ the elements $g_1,\dots,g_m$ have trivial images under the  homomorphism $i_{WV}:H_k(X,X\setminus\overline{W})\to H_k(X,X\setminus \bar V)$. Since these elements generate the group $\mathrm{Im}(i_{UW})$, the inclusion homomorphism $i_{UV}=i_{WV}\circ i_{UW}:H_k(X,X\setminus \bar U)\to H_k(X,X\setminus \bar V)$ is trivial.
\end{proof}

\section{Characterizing locally $n$-negligible sets}

In this section we present a homological characterization of locally $n$-negligible sets. Following H.~Toru\'n\-czyk \cite{To} we define a subset $A$ of a space $X$ to be {\em locally $n$-negligible} if given $x\in X$, $k<n+1$, and a neighborhood $U$ of $x$ there is a neighborhood $V\subset U$ of $x$ such that for each $f:(I^k,\partial I^k)\to (V,V\setminus A)$ there is a homotopy $(h_t):(I^k,\partial I^k)\to (U,U\setminus A)$ with $h_0=f$ and $h_1(I^k)\subset U\setminus A$. Following \cite{Spa} we shall say that a pair $(X,A)$ is {\em $n$-connected} if $\pi_i(X,A)=0$ for all $i\le n$.

\begin{theorem}\label{torunczyk} For a subset $A$ of a Tychonov space $X$ the following conditions are equivalent:
\begin{enumerate}
\item $A$ is locally $n$-negligible;
\item given: a simplicial pair $(K,L)$ with $\dim(K)\le n$, a continuous pseudometric $\rho$ on $X$ and maps $\e:|K|\to(0,\infty)$ and $f:|K|\times\{0\}\cup|L|\times I\to X$ with $\rho(f(x,0),f(x,t))<\e(x)$ and $f(x,1)\notin A$ for all $(x,t)\in|L|\times I$ there is $\bar f:|K|\times I\to X$ which extends $f$ and satisfies $\rho(\bar f(x,t),f(x,0))<\e(x)$ and $\bar f(x,1)\notin A$ for all $(x,t)\in|K|\times I$;
\item for each open set $U\subset X$ and $k<n+1$ the relative homotopy group $\pi_k(U,U\setminus A)$ vanishes;
\item each $x\in X$ has a basis $\mathfrak U_x$ of open neighborhoods with $\pi_k(U,U\setminus A)=0$ for all $U\in\mathfrak U_x$ and $k<n+1$.
\smallskip

\hskip-45pt If, in addition,  $X$ is an $\LC[1]$-space and $A$ is
locally $2$-negligible in $X$, then the conditions \emph{(1)--(4)}
are equivalent to
\smallskip

\item for each open $U\subset X$ and $k<n+1$ the relative homology group  $H_k(U,U\setminus A)$ vanishes;
\item each $x\in X$ has a basis $\mathfrak U_x$ of open neighborhoods with $H_k(U,U\setminus A)=0$ for all $U\in\mathfrak U_x$ and $k<n+1$.
\end{enumerate}
\end{theorem}

\begin{proof} The equivalence $(1)\Leftrightarrow(2)\Leftrightarrow(3)\Leftrightarrow(4)$ have
been proved by Toru\'nczyk \cite[2.3]{To} in case of a normal space $X$.
But because of Lemma~\ref{metric} his proof works also for any Tychonov $X$.
\smallskip

Now assume that  $A$ is a locally $2$-negligible set in an $\LC[1]$-space $X$.
\smallskip

The implication $(5)\Ra(6)$ is trivial while $(3)\Ra(5)$ easily follows from the Hurewicz Isomorphism Theorem 4 in \cite[\S 7.4]{Spa} (see also Theorem 4.37 \cite{Hat}) because $X$ is an $\LC[0]$-space and $A$ is  locally 1-negligible in $X$.
\smallskip

It remains to prove the implication  $(6)\Ra (1)$.
Take any point $x\in X$ and a neighborhood $U\subset X$ of $x$. We lose no generality assuming that $U$ is connected.

The $\LC[1]$-property of $X$ yields a neighborhood $V\subset U$ of $x$ such that
any map $f:\partial I^2\to V$ extends to a map $\bar f:I^2\to U$. Moreover, by (6) we can choose the neighborhood $V$ so that $H_k(V,V\setminus A)=0$ for all $k<n+1$. Replacing $V$ by a connected component containing the point $x$, if necessary, we may assume that $V$ is connected. Then the $\LC[0]$-property of $X$ implies that  $V$ is path connected and the local 1-negligibility of $A$ in $X$ implies that $V\setminus A$ is path-connected too.

We claim that every map $f:(I^k,\partial I^k)\to (V,V\setminus A)$ with $k<n+1$ is homotopic in $U$ to a map $\tilde f:(I^k,\partial I^k)\to (U\setminus A,U\setminus A)$ which will imply the local $n$-negligibility of $A$ in $X$.  This will be done by induction. For $k\le 2$ the local $k$-negligibility of $A$ in $X$ follows from the local 2-negligibility and there is nothing to prove. So we assume that the local $(k-1)$-negligibility of $A$ has been proved for some $k>2$. Then the implication $(1)\Ra(3)$ yields $\pi_i(V,V\setminus A)=0$ for all $i< k$. This means that the pair $(V,V\setminus A)$ is $(k-1)$-connected, which makes legal applying the Relative Hurewicz Isomorphism Theorem 4 from \cite[\S7.5]{Spa} to this pair.

Fix any point $*\in\partial I^k$ and let $x_0=f(*)$. Then the map $f$ can be considered as an element of the relative homotopy group $\pi_k(V,V\setminus A,*)$. The Relative Hurewicz Isomorphism Theorem applied to the pair $(V,V\setminus A)$ implies that $\pi_k'(V,V\setminus A,x_0)=0$, where $\pi_k'(V,V\setminus A,x_0)$ is the quotient group of the homotopy group  $\pi_k(V,V\setminus A,x_0)$ by the normal subgroup $G$ generated by the elements $[\gamma]\cdot[\alpha]-[\alpha]$, $[\alpha]\in \pi_k(V,V\setminus A,x_0)$, $[\gamma]\in\pi_1(V\setminus A,x_0)$, where $\cdot:\pi_1(V\setminus A,x_0)\times \pi_k(V,V\setminus A,x_0)\to \pi_k(V,V\setminus A,x_0)$ is the left action of the fundamental group on the relative $k$-th homotopy group, see \cite[\S 7.3]{Spa} or \cite{Hat}.
It follows from $0=\pi'_k(V,V\setminus A,x_0)=\pi_k(V,V\setminus A,x_0)/G$ that the subgroup $G$ coincides with $\pi_k(V,V\setminus A,x_0)$. Now consider the homomorphism $i_*:\pi_k(V,V\setminus A,x_0)\to\pi_k(U,U\setminus A,x_0)$ induced by the inclusion of pairs $i:(V,V\setminus A)\to (U,U\setminus A)$.
We claim that $i_*$ is a null-homomorphism. Since the group $\pi_k(V,V\setminus A,x_0)$ is generated by the elements $[\gamma]\cdot[\alpha]-[\alpha]$, $[\alpha]\in \pi_k(V,V\setminus A,x_0)$, $[\gamma]\in\pi_1(V\setminus A,x_0)$, it suffices to check that $i_*([\gamma]\cdot[\alpha]-[\alpha])=0$ for any such $[\alpha],[\gamma]$.
Let $j_*:\pi_1(V\setminus A,x_0)\to \pi_1(U\setminus A,x_0)$ be the homomorphism induced by the inclusion $j:(V,x_0)\to (U,x_0)$.
The local 2-negligibility of $A$ in $U$ and the choice of the set $V$ implies that $j_*=0$ and thus the action of the element $j_*([\gamma])\in\pi_1(U\setminus A,x_0)$ on $\pi_k(U,U\setminus A,x_0)$ is trivial. The naturality of the action of the fundamental group on the homotopy groups (see Lemma~1 in \cite[\S7.3]{Spa}) implies that $$i_*([\gamma]\cdot[\alpha]-[\alpha])=j_*([\gamma])\cdot i_*([\alpha])-i_*([\alpha])=i_*([\alpha])-i_*([\alpha])=0.$$

Thus the homomorphism $i_*:\pi_k(V,V\setminus A,x_0)\to \pi_k(U,U\setminus A,x_0)$ is trivial, which means that the map $f:(I^k,\partial I^k,*)\to(V,V\setminus A,x_0)$ is null homotopic in $(U,U\setminus A,x_0)$ and this completes the proof of the local $k$-negligibility of $A$ in $X$.
\end{proof}

\section{Homotopical and Homological $Z_n$-sets}\label{s:Zsets}

In this section we apply the characterization of locally $n$-negligible sets established in the preceding section to study $Z_n$-sets. We recall the definition of a $Z_n$-set and  its homotopical and homological versions.

\begin{definition} A closed subset $A$ of a topological space $X$ is defined to be
\begin{itemize}
\item a {\em $Z_n$-set} if for any open cover $\U$ of $X$ and any map $f:I^n\to X$ there is a map $f':I^n\to X\setminus A$ which is $\U$-near to $f$;
\item a {\em homotopical $Z_n$-set} in $X$ if for any open cover $\U$ of $X$ and any map $f:I^n\to X$ there is a map $f':I^n\to X\setminus A$ which is $\U$-homotopic to $f$;
\item a {\em $G$-homological $Z_n$-set} in $X$, where $G$ is a coefficient group, if for any open set $U\subset X$ and any $k<n+1$ the relative homology group $H_k(U,U\setminus A;G)=0$;
\item a {\em homological $Z_n$-set} in $X$ if it is a $\IZ$-homological $Z_n$-set in $X$.
\end{itemize}
\end{definition}

A point $x$ in a space $X$ is called a {\em $Z_n$-point} if the singleton $\{x\}$ is a $Z_n$-set in $X$. By analogy we define homotopical and homological $Z_n$-points.
\smallskip

The following theorem reveals interplay between various versions of $Z_n$-sets.

\begin{theorem}\label{Zsets} Let $G$ be a non-trivial Abelian group and $X$ be a Tychonov space.
\begin{enumerate}
\item A subset $A$ of $X$ is a homotopical $Z_n$-set in $X$ iff $A$ is closed and locally $n$-negligible set in $X$.
%\smallskip

\item Each homotopical $Z_n$-set in $X$ is a $Z_n$-set in $X$.
\item If $X$ is an $\LC[n]$-space, then each $Z_n$-set in $X$ is a homotopical $Z_n$-set.
%\smallskip

\item Each homotopical $Z_n$-set in $X$ is a $G$-homological $Z_n$-set.
\item Each $G$-homological $Z_0$-set in $X$ is a homotopical $Z_0$-set.
\item Each $G$-homological $Z_1$-set in $X$ is a $Z_1$-set.
\end{enumerate}
\end{theorem}

\begin{proof} The first item follows immediately from Theorem~\ref{torunczyk} while the second is trivial. The third item was proved in \cite[3.3]{To} and follows from Lemma~\ref{LCnear}. The fourth item can be easily derived from Corollary 10.6 \cite{Bre2} (of  the Relative Hurewicz Theorem). The fifth item follows from the fact that a relative homology group $H_0(U,U\setminus A;G)$ vanishes if and only if each path-connected component of $U$ meets the set $U\setminus A$.
\smallskip

To prove the sixth item assume that $A$ is a $G$-homological $Z_1$-set in $X$.
Being a $G$-homological $Z_0$-set, $A$ is a homotopical $Z_0$-set in $X$.
To show that $A$ is a $Z_1$-set in $X$, fix an open cover $\U$ of $X$ and a map $f:I\to X$. Consider the open cover $f^{-1}(\U)=\{f^{-1}(U):U\in\U\}$ of the interval $I=[0,1]$ and find a sequence $0=t_0<t_1<\dots<t_m=1$ such that for each $i\le m$ the interval $[t_{i-1},t_i]$ lies in $f^{-1}(U_i)$ for some set $U_i\in\U$. Let $W_m=U_m$ and  $W_i=U_{i}\cap U_{i+1}$ for $i<m$.

Since $H_0(W_i,W_i\setminus A;G)=0$, the path-connected component
of $W_i$ containing the point $f(t_i)$ meets the set $W_i\setminus
A$ at some point $x_i$. We claim that the points $x_{i-1},x_i$ lie
in the same path-connected component of $U_i\setminus A$. Assuming
the converse we would get a nontrivial cycle $\alpha=g\cdot
x_{i-1}-g\cdot x_i$ in $H_0(U_i\setminus A;G)$ with $g\in G$ being
any non-zero element. On the other hand, this cycle is the
boundary of an obvious 1-chain $\beta$ in $U_i$ and thus vanishes
in the homology group $H_0(U_i;G)$. But this contradicts the exact
sequence $$0=H_1(U_i,U_i\setminus A;G)\to H_0(U_i\setminus A;G)\to
H_0(U_i;G)$$ for the pair
$(U_i,U_i\setminus A)$.

Therefore $x_{i-1},x_i$ lie in the same path-connected component
of $U_i\setminus A$, ensuring the existence of a continuous map
$g_i:[t_{i-1},t_i]\to U_i\setminus A$ with $g_i(t_{i-1})=x_{i-1}$
and $g_i(t_i)=x_i$. The maps $g_i$, $i\le m$, compose a single
continuous map $g:[0,1]\to X\setminus A$ which is $\U$-near to
$f$, witnessing the $Z_1$-set property of $A$.
\end{proof}

Combining the first item of Theorem~\ref{Zsets} with Theorem~\ref{torunczyk}, we get the following important characterization of homotopical $Z_n$-sets.

\begin{theorem}\label{Z2+hZn=Zn} A homotopical $Z_2$-set $A$ in a  Tychonov $\LC[1]$-space $X$ is a homotopical $Z_n$-set in $X$ if and only if $A$ is a homological $Z_n$-set in $X$.
\end{theorem}

\begin{remark} Theorem~\ref{Z2+hZn=Zn} has been known as a folklore and its particular cases have appeared in literature,
see e.g. \cite{Kro}, \cite{Dob}, \cite{DW}. It should be also mentioned that a substantial part of \cite{DW} is devoted to closed sets of infinite codimension (coinciding with our homological $Z_\infty$-sets).
\end{remark}

Next, we establish some elementary properties of $G$-homological $Z_n$-sets. From now on,  $G$ is a non-trivial abelian group.

\begin{proposition}\label{subset} Let $A,B$ be $G$-homological $Z_n$-sets in a topological space $X$.
\begin{enumerate}
\item Any closed subset $F\subset A$ is a $G$-homological $Z_n$-set in $X$.
\item The union $A\cup B$ is a $G$-homological $Z_n$-set in $X$.
\end{enumerate}
\end{proposition}

\begin{proof} 1. We should check that $H_k(U,U\setminus F;G)=0$ for all open sets $U\subset X$ and all $k<n+1$. The $G$-homology $Z_n$-set property of $A$ yields $H_k(U,U\setminus A;G)=0$.
Since $(U\setminus F,U\setminus A)=(U\setminus F,(U\setminus F)\setminus A)$, we get also  $H_{k-1}(U\setminus F,U\setminus A;G)=0$. Writing the exact sequence
of the triple $(U,U\setminus F,U\setminus A)$ gives us $H_k(U,U\setminus F;G)=0$.
\smallskip

2. Given any $k<n+1$ and any open set $U\subset X$ consider the exact sequence of the triple \mbox{$(U,U\setminus A,U\setminus(A\cup B))$}:
{\small $$
0=H_k(U\setminus A, U\setminus (A\cup B);G)\to H_k(U,U\setminus(A\cup B);G)\to H_k(U,U\setminus A;G)=0$$}
and conclude that $H_k(U,U\setminus(A\cup B);G)=0$.
\end{proof}

Next, we show that in the definition of a $G$-homological $Z_n$-set we can require that $U$ runs over some base for $X$.

\begin{proposition}\label{mult} Let $\mathcal B$ be a base of the topology
of a topological space $X$. A closed set $A\subset X$ is a $G$-homological $Z_n$-set
in $X$ if and only if $H_k(U,U\setminus A;G)=0$ for all $k<n+1$
and $U\in\mathcal B$.
\end{proposition}

\begin{proof} The ``only if'' part of the theorem is trivial. To prove
the ``if'' part, assume that $H_k(U,U\setminus A;G)=0$ for all
$k<n+1$ and all sets $U\in\mathcal B$. Let $\U_k$ be the family of
all open subsets $U\subset X$ such that $H_k(U,U\setminus A;G)=0$.
By induction on $k<n+1$ we shall show that $\U_k$ consists of all
open sets in $X$.

First we verify the case $k=0$. Take any non-empty open set
$U\subset X$. The equality  $H_0(U,U\setminus A;G)=0$ will follow
as soon as we show that each path-connected component $C$ of $U$
intersects the set $U\setminus A$. Fix any point $c\in C$ and find
a neighborhood $V\in\mathcal B$ of $x$ lying in $U$. The equality
$H_0(V,V\setminus A;G)=0$ implies that $V\setminus A$ meets each
path-connected component of the set $V$, in particular the
path-connected component $C'$ of the point $c$ in $V$. Since
$C'\subset C$, we conclude that $U\setminus A\supset V\setminus A$
meets the component $C$, which completes the proof of the case
$k=0$.

Assuming that for some positive $k<n+1$ the family $\U_{k-1}$
consists of all open subsets of $X$, we first show that the family
$\U_k$ is closed under finite unions. Indeed, given any two sets
$U,V\in\U_k$ we can write down the piece of the relative
Mayer-Vietoris exact sequence $$
H_k(U,U\setminus A;G)\oplus H_k(V,V\setminus A;G)\to H_k(U\cup V,(U\cup V)\setminus A;G)\to H_{k-1}(U\cap V,U\cap V\setminus A;G)
$$
and conclude that $U\cup V\in\U_k$ because $U\cap V\in\U_{k-1}$.
Since $\U_k$ contains the base $\mathcal B$, it contains all possible finite unions of sets of the base. Because the singular homology theory has compact support, $\U$ contains all possible unions of sets from the base $\mathcal B$ and consequently, $\U_k$ consists of all open sets in $X$.
\end{proof}

\begin{corollary}\label{local} A subset $A$ of a topological space $X$ is a $G$-homolo\-gical $Z_n$-set in $X$ if and only if there is an open cover $\U$ of $X$ such that for every $U\in\U$ the intersection $U\cap A$ is a $G$-homological $Z_n$-set in $U$.
\end{corollary}

\section{Detecting $G$-homological $Z_n$-sets with help of partitions}\label{irred}

In this section we apply the technique of irreducible homological barriers to detect $G$-homological $Z_n$-sets with help of their partitions.

A closed subset $B$ of a topological space $X$ is called an {\em irreducible barrier} for a non-zero homology element $\alpha\in H_n(X,X\setminus B;G)$ if for every closed subset $A\subset B$ with $A\ne B$ the image $i^B_A(\alpha)$ under the inclusion homomorphism $i^B_A:H_n(X,X\setminus B;G)\to H_n(X,X\setminus A;G)$ is trivial.  We shall say that a subset $A$ of a topological space $X$ is {\em separated} by a subset $B\subset X$ if the complement $A\setminus B$ is disconnected. In Dimension  Theory closed separating sets also are referred to as {\em partitions}, see \cite[1.1.3]{En}.

We shall need two elementary properties of irreducible barriers, which were also exploited in  \cite{Cat} and \cite{BC01}.

\begin{lemma}\label{irreducible} Let $X$ be a topological space and $G$ be a non-trivial abelian group.
\begin{enumerate}
\item Each closed subset $A$ of $X$ with $H_n(X,X\setminus A;G)\ne 0$ for some $n\ge 0$ contains an irreducible barrier $B\subset A$ for some element $\alpha\in H_n(X,X\setminus B;G)$.
\item If $A$ is an irreducible barrier for some element $\alpha\in H_n(X,X\setminus A;G)$, then $H_{n+1}(X,X\setminus B;G)\ne0$ for any closed subset $B\subset A$ separating $A$.
\end{enumerate}
\end{lemma}

\begin{proof} 1. Given a closed subset $A\subset X$ and a non-zero element $\alpha\in H_n(X,X\setminus A;G)\ne0$, consider the family $\mathcal B$ of closed subsets of $A$ such that for every $B\in\mathcal B$ the image $i^A_B(\alpha)$ under the inclusion homomorphism $i^A_B:H_n(X,X\setminus A;G)\to H_n(X,X\setminus B;G)$ is not trivial. We claim that $\mathcal B$ contains a minimal element. This will follow from Zorn Lemma as soon as we prove that the intersection $\cap\mathcal C$ of any linearly ordered subfamily $\C\subset\mathcal B$ belongs to $\mathcal B$. Since singular homology has compact support, the homology group $H_n(X,X\setminus \cap \C;G)$ is the direct limit of the groups $H_n(X,X\setminus C;G)$, $C\in\C$, see \cite[IV.\S8.13]{Bre2}. Since all the elements $i^A_C(\alpha)$, $C\in\C$, are not trivial, so is the element $i^A_{\cap C}(\alpha)$, which means that $\cap\C\in\mathcal B$. Now, the Zorn Lemma yields a minimal element $B$ in $\mathcal B$. Let $\beta=i^A_B(\alpha)\ne 0$. The minimality of $B$ implies that for any closed subset $C\subset B$ with $C\ne B$ we get $i^A_C(\alpha)=i^B_C(\beta)=0$, which means that $B$ is an irreducible barrier for $\beta$.
\smallskip

2. Assume that $A$ is an irreducible barrier for some element $\alpha\in H_n(X,X\setminus A;G)$ and let $B$ be a closed subset separating $A$. Write $A\setminus B=U\cup V$ as the union of two disjoint non-empty open sets $U,V\subset A$. Let $C=V\cup B=A\setminus U$ and $D=U\cup B=A\setminus V$. The irreducibility of $A$ for $\alpha$ yields $i^A_{C}(\alpha)=i^A_{D}(\alpha)=0$. Now writing a Mayer-Vietoris exact sequence for
the pair $(X,X\setminus B)=(X,X\setminus C\cup X\setminus D)$, we get
{\footnotesize $$
H_{n+1}(X,X\setminus B;G)\overset{\partial}{\longrightarrow} H_n(X,X\setminus A;G)\overset{f}{\longrightarrow} H_n(X,X\setminus C;G)\oplus H_n(X,X\setminus D;G).$$}
Since $f(\alpha)=(i^A_C(\alpha),-i^A_D(\alpha))=(0,0)$, $0\ne\alpha=\partial(\beta)$ for some nontrivial element $\beta\in H_{n+1}(X,X\setminus B;G)$.
\end{proof}

Irreducible barriers help us to prove the following theorem detecting $G$-homological $Z_n$-sets.

 \begin{theorem}\label{separate} A closed subset $A$ of a topological space $X$ is a $G$-homological $Z_n$-set in $X$ if each point of $A$ is a $G$-homological $Z_n$-point in $X$ and each closed subset $B$ of $A$ with $|B|>1$ can be separated by a $G$-homological $Z_{n+1}$-set.
\end{theorem}

\begin{proof} Assume that each point of a closed subset $A\subset X$ is a
$G$-homo\-logical $Z_{n}$-point and each closed subset $B\subset A$ with $|B|>1$ can be separated by a $G$-homological $Z_{n+1}$-set $B\subset A$ in $X$. Assuming that $A$ fails to be a $G$-homological $Z_n$-set, find an open subset $U\subset X$ and $k<n+1$ with $H_k(U,U\setminus A;G)\ne0$. By Lemma~\ref{irreducible}(1), the set $A\cap U$ contains an irreducible barrier $B$ for a non-zero element $\alpha\in H_k(U,U\setminus B;G)$. The set $B$ must contain more than one point, since singletons are $G$-homological $Z_n$-sets in $X$. By our assumption the closure $\overline{B}$ in $X$ can be separated by a $G$-homological $Z_{n+1}$-set $C$. Then $C\cap U$ separates $\overline{B}\cap U=B$ and hence $H_{k+1}(U,U\setminus C;G)\ne0$ by Lemma~\ref{irreducible}(2). But this is not possible because $C$ is a $G$-homological $Z_{n+1}$-set in $X$.
\end{proof}

Theorem~\ref{separate} will be applied to show that a subset
$A\subset X$ is a $G$-homological $Z_n$-set in $X$ if each point $a\in A$ is a $Z_{n+d}$-point in $X$ where $d=\trt(A)$ is the separation dimension of $A$.
The separation dimension $\mathsf{t}(\cdot)$ was introduced by G.~Steinke \cite{Sk} and later was extended to the transfinite separation dimension $\trt(\cdot)$ by Arenas, Chatyrko and Puertas \cite{ACP} as follows: given a topological space $X$  we write
\begin{itemize}
\item $\trt(X)=-1$ iff $X=\emptyset$;
\item $\trt(X)\le\alpha$ for an ordinal $\alpha$ if any closed subset $B\subset X$ with $|B|\ge2$ can be separated by a closed subset $P\subset B$  with $\trt(P)<\alpha$.
\end{itemize}
A space $X$ is defined to be {\em $\trt$-dimensional} if $\trt(X)\le\alpha$ for some ordinal $\alpha$. In this case the ordinal
$$\trt(X)=\min\{\alpha:\trt(X)\le\alpha\}$$ is called the (transfinite) separation dimension of $X$. If $X$ is not $\trt$-dimensional, then we write $\trt(X)=\infty$ and assume that $\alpha<\infty$ for all ordinals $\alpha$.

By transfinite induction one can show that $\trt(X)\le\mathrm{trind}(X)$ where $\mathrm{trind}(X)$ is the transfinite extension of the small inductive dimension $\mathrm{ind}(X)$, see \cite[2.9]{ACP}. This implies that each countable-dimensional completely-metrizable space is $\trt$-dimensional (because $\mathrm{trind}(X)<\omega_1$ \cite[7.1.9]{En}). On the other hand, each $\trt$-dimensional compact space is a $C$-space, see \cite[4.7]{ACP}.
Observe that $\trt(X)\le0$ if and only if $X$ is hereditarily disconnected space.
The strongly infinite-dimensional totally disconnected Polish space $X$ constructed in \cite[6.2.4]{En} has $\trt(X)=0$ and $\trind(X)=\infty$ while a compatification $c(X)$ of $X$ with strongly countable-dimensional remainder (the famous Pol's example) has $\trt(X)=\w$ and $\mathrm{trind}(X)=\infty$. Thus (even on the compact level) the gap between $\trt(X)$ and $\mathrm{trind}(X)$ can be huge. However, for finite-dimensional metrizable compacta $X$, the separation dimension $\trt(X)$ coincides with the usual dimension $\dim(X)$ (and hence with $\mathrm{trind}(X)$), see \cite{Sk}.

\begin{theorem}\label{indP1} A closed subspace $A$ of a space $X$ with $m=\trt(A)<\w$ is a $G$-homological $Z_n$-set in $X$ provided each point $a\in A$ is a $G$-homological $Z_{n+m}$-point in $X$.
\end{theorem}

\begin{proof} The proof is by induction of the number $m=\trt(A)$.
The assertion is trivial if $m=-1$ (which means that $A$ is empty). Assume that for some number $m$ the theorem has been proved for all sets $A$ with $\trt(A)<m$. Take a closed subset $A\subset X$ with $\trt(A)=m$ and all points $a\in A$ being $G$-homological $Z_{n+m}$-points in $X$. Assuming that $A$ fails to be a $Z_n$-set in $X$, find an open set $U\subset X$ and a number $k<n+1$ with $H_k(U,U\setminus A;G)\ne0$. By Lemma~\ref{irreducible}(1), the set $A\cap U$ contains an irreducible barrier $B\subset A\cap U$ for some non-zero element $\beta\in H_k(U,U\setminus B;G)$. Since singletons are $G$-homological $Z_n$-sets in $X$, $|B|>1$. Since $\trt(A)\le m$, there is a partition $C\subset B$ of $B$ with $d=\trt(C)<m$. Then all points of the set $C$ are $G$-homological $Z_{n+d+1}$-points. Now the inductive assumption guarantees that $C$ is a $G$-homological $Z_{n+1}$-set in $U$, which contradicts  Lemma~\ref{irreducible}(2) because $C$ separates the irreducible barrier $B$.
\end{proof}

By the same method one can prove an infinite version of this theorem.

\begin{theorem}\label{trans} A closed $\trt$-dimensional subspace $A$ of a space $X$ is  a $G$-homological $Z_\infty$-set in $X$ if and only if each point $a\in A$ is a $G$-homological $Z_\infty$-point in $X$.
\end{theorem}

\section{Bockstein Theory for $G$-homological $Z_n$-sets}\label{bock}

In this section, we study the interplay between $G$-homological $Z_n$-sets for various coefficient groups $G$.
Our principal instrument here is the Universal Coefficients Formula expressing the homology with respect to an arbitrary coefficient group via homology with respect to the group $\IZ$ of integers.
The following its form is taken from \cite[3A.4]{Hat}.

\begin{lemma}[Universal Coefficients Formula]\label{univcoef} For each pair $(X,A)$ and all $n\ge 1$ there is a natural exact sequence
$$0\to H_n(X,A)\otimes G\to H_n(X,A;G)\to H_{n-1}(X,A)*G\to 0$$and this sequence splits (though non-naturally).
\end{lemma}

Here $G\otimes H$ and $G*H$ stand for the tensor and torsion products of the groups $G,H$, respectively. We need some information of those products.

At first some notation. By $\Pi$ we denote the set of prime numbers. We recall that a group $G$ is {\em divisible} if it is {\em divisible by each number $n\in\IN$}. The latter means that for any $g\in G$ there is $x\in G$ with $n\cdot x=g$. By $\Tor(G)=\{x\in G:\exists n\in\IN$ with $nx=0\}$ we denote the {\em torsion part} of $G$. It is well-known that $\Tor(G)$ is the direct sum of $p$-torsion parts
$$\pTor(G)=\{x\in G:\exists k\in\IN\;\; p^kx=0\}$$where $p$ runs over prime numbers.

The following useful result can be found in \cite{Dyer}.

\begin{lemma}\label{dyer} Let $G_0,G_1$ be non-trivial abelian groups and $p$ be a prime number.
\begin{enumerate}
\item The tensor product $G_0\otimes G_1$ contains an element of infinite order if and only if both groups $G_0$ and $G_1$ contain elements of infinite order.
\item The torsion product $G_0*G_1$ contains an element of order $p$ if and only if $G_0$ and $G_1$ contain elements of order $p$.
\item The tensor product $G_0\otimes G_1$ contains an element of order $p$ if and only if for some $i\in\{0,1\}$ either $G_i$ is not divisibe by $p$ and $\pTor(G_{1-i})$ is not divisible by $p$ or else $G_i/\pTor(G_i)$ is not divisible by $p$ and $\pTor(G_{1-i})\ne0$ is divisible by $p$.
\end{enumerate}
\end{lemma}

The following fact follows immediately from the Universal Coefficients Formula.

\begin{proposition}\label{Z->G} Each homological $Z_n$-set in a space $X$ is a $G$-homological $Z_n$-set in $X$ for any coefficient group $G$.
\end{proposition}

Next, we show that the study of $G$-homological $Z_n$-sets for an arbitrary coefficient group $G$ can be reduced to studying $H$-homological $Z_n$-sets for coefficient groups $H$ from the countable family of so-called Bockstein groups.

The are many (more or less standard) notations for Bockstein groups. We follow those of \cite{Kuz} and \cite{Dyer}:
\begin{itemize}
\item $\IQ$ is the group of rational numbers;
\item $\IZ_p=\IZ/p\IZ$ is the cyclic group of a prime order $p$;
\item $\Rp$ is the group of rational numbers whose denominator is not divisible by $p$.
\item $\IQ_p=\IQ/\Rp$ the quasicyclic $p$-group.
\end{itemize}
By $\mathfrak B=\{\IQ,\IZ_p,\IQ_p,\Rp:p\in\Pi\}$ we denote the family of all Bockstein groups.
\smallskip

To each abelian group $G$ assign the Bockstein family $\sigma(G)\subset \mathfrak B$
containing the group:
\begin{itemize}
\item $\IQ$ if and only if the group $G/\Tor(G)\ne0$ is divisible;
\item $\Rp$ if and only if $G/\pTor(G)$ is not divisible by $p$;
\item $\IZ_p$ if and only if $\pTor(G)$ is not divisible by $p$;
\item $\IQ_p$ if and only if $\pTor(G)\ne0$ is divisible by $p$.
\end{itemize}
In particular, $\sigma(G)=\{G\}$ for every Bockstein group $G\in\mathfrak B$.

The Bockstein families are helpful because of the following fact which can be easily derived from Lemma~\ref{dyer}:

\begin{lemma}\label{dyer2} For abelian groups $G,H$ the tensor product $G\otimes H$ is trivial if and only if $G\otimes B$ is trivial for every group $B\in\sigma(H)$.
\end{lemma}

Combining this lemma with the Formula of Universal Coefficients we will obtain the principal result of this section.

\begin{theorem}\label{bockstein} A closed subset $A$ of a topological space $X$ is a $G$-homological $Z_n$-set in $X$ if and only if $A$ is an $H$-homological $Z_n$-set in $X$ for every group $H\in\sigma(G)$.
\end{theorem}

\begin{proof} Assume that a closed subset $A\subset X$  fails to be a $G$-homologi\-cal $Z_n$-set in $X$ and find $k<n+1$ and an open set $U\subset X$ with $H_k(U,U\setminus A;G)\ne0$. The Universal Coefficients Formula implies that either $H_k(U,U\setminus A)\otimes G\ne 0$ or $H_{k-1}(U,U\setminus A)*G\ne 0$.

In the latter case, the group
$H_{k-1}(U,U\setminus A)*G$ contains an element of a prime order $p$ and then both the groups $H_{k-1}(U,U\setminus A)$ and $G$ contain elements of order $p$ by Lemma~\ref{dyer}(2). Consequently, $\sigma(G)$ contains a (quasi)cyclic $p$-group $H\in\{\IZ_p,\IQ_p\}$ and then $H_{k-1}(U,U\setminus A)*H\ne0$. Applying the Formula of Universal Coefficients, we conclude that $H_k(U,U\setminus A;H)\ne0$.

Next, we assume that $H_k(U,U\setminus A)\otimes G\ne0$. Then Lemma~\ref{dyer}  implies that $H_k(U,U\setminus A)\otimes H\ne0$ for some group $H\in\sigma(G)$. Applying the Formula of Universal Coefficients, we conclude that $H_k(U,U\setminus A;H)\ne0$, which means that $A$ fails to be an $H$-homolo\-gical $Z_n$-set in $X$. This proves the ``if'' part of the theorem.
\smallskip

To prove the ``only if'' part, assume that $A$ fails to be a $H$-homological $Z_n$-set in $X$ for some group $H\in\sigma(G)$. Then for some $k<n+1$ and an open set $U\subset X$ the group $H_k(U,U\setminus A;H)$ is not trivial. The Formula of Universal Coefficients yields that either $H_k(U,U\setminus A)\otimes H\ne0$ or $H_{k-1}(U,U\setminus A)*H\ne0$.

In the first case the tensor product $H_k(U,U\setminus A)\otimes G\ne0$ by Lemma~\ref{dyer}. In the second case, $H_{k-1}(U,U\setminus A)*H$ contains an element of a prime order $p$ and so do the groups $H_{k-1}(U,U\setminus A)$ and $H$. The inclusion $H\in\sigma(G)$ implies then that $G$ contains an element of order $p$ and hence $H_{k-1}(U,U\setminus A)*G\ne0$. In both cases the Formula of Universal Coefficients implies that $H_k(U,U\setminus A;G)\ne0$, which means that $A$ fails to be a $G$-homological $Z_n$-set in $X$.
\end{proof}

Next, we study the interplay between $G$-homological $Z_n$-sets
for various Bockstein groups $G$.

\begin{theorem}\label{bockstein2} Let $A$ be a closed subset of a space $X$ and $p$ be a prime number.
\begin{enumerate}
\item If $A$ is a $\Rp$-homological $Z_n$-set in $X$, then $A$ is a $\IQ$-homological and $\IZ_p$-homological $Z_n$-set in $X$.
\item If $A$ is a $\IZ_p$-homological $Z_n$-set in $X$, then $A$ is a $\IQ_p$-homological $Z_n$-set in $X$.
\item If $A$ is a $\IQ_p$-homological $Z_{n+1}$-set in $X$, then $A$ is a $\IZ_p$-homological $Z_n$-set in $X$.
\item $A$ is a $\Rp$-homological $Z_n$-set in $X$ provided $A$ is a $\IQ$-homological $Z_n$-set in $X$ and a $\IQ_p$-homological $Z_{n+1}$-set in $X$.
\end{enumerate}
\end{theorem}

\begin{proof} 1. Assuming that $A$ is not a $\IQ$-homological $Z_n$-set in $X$, find a number $k<n+1$ and an open set $U\subset X$ with $H_k(U,U\setminus A;\IQ)\ne0$. Since $\IQ$ is torsion free, the formula of universal coefficients implies that $H_k(U,U\setminus A)$ contains an element of infinite order and then $H_k(U,U\setminus A)\otimes\Rp\ne0$. Applying the Formula of Universal Coefficients once more, we obtain that $H_k(U,U\setminus A;\Rp)\ne0$, which means that $A$ is not a $\Rp$-homological $Z_n$-set in $X$.

Next, assume that $A$ is not a $\IZ_p$-homological $Z_n$-set in $X$ and find an integer number $k\le n$ and an open set $U\subset X$ with  $H_k(U,U\setminus A;\IZ_p)\ne0$. The Formula of Universal Coefficients implies that either $H_k(U,U\setminus A)\otimes\IZ_p\ne0$ or $H_{k-1}(U,U\setminus A)*\IZ_p\ne0$. In the latter case, the group $H_{k-1}(U,U\setminus A)$ contains an element of order $p$ and by Lemma~\ref{dyer}, $H_{k-1}(U,U\setminus A)\otimes\Rp\ne0$. Now the Formula of Universal Coefficients implies that $H_{k-1}(U,U\setminus A;\Rp)\ne0$, which means that $A$ fails to be a $\Rp$-homological $Z_{k-1}$-set in $X$.

So assume that $H_k(U,U\setminus A)\otimes\IZ_p\ne0$. If $H_k(U,U\setminus A)$ contains an element of order $p$, then we can proceed as in the preceding case to prove that $A$ fails to be a $\Rp$-homological $Z_k$-set in $X$. So we can assume that $\pTor(H_k(U,U\setminus A))=0$, which implies that the torsion part $\Tor(H_k(U,U\setminus A))$ is divisible by $p$. Taking into account that $H_k(U,U\setminus A)\otimes\IZ_p\ne0$, we can apply Lemma~\ref{dyer}(3) to find an element $a\in H_k(U,U\setminus A)$, not divisible by $p$. This element cannot belong to the torsion part of $H_k(U,U\setminus A)$. Hence the group $H_k(U,U\setminus A)$ contains an element of infinite order and so does the tensor product $H_k(U,U\setminus A)\otimes\Rp=H_k(U,U\setminus A;\Rp)$. This means that $A$ fails to be a $\Rp$-homological $Z_k$-set in $X$.
\smallskip

2. Now assume that $A$ fails to be a $\IQ_p$-homological $Z_n$-set in $X$ and find $k\le n$ and an open set $U\subset X$ with \mbox{$H_k(U,U\setminus A;\IQ_p)\ne0$}. Applying the Formula of Universal Coefficients, we get  $H_k(U,U\setminus A)\otimes\IQ_p\ne0$ or $H_{k-1}(U,U\setminus A)*\IQ_p\ne0$. In the latter case, $H_{k-1}(U,U\setminus A)*\IZ_p\ne0$ and hence $H_k(U,U\setminus A;\IZ_p)\ne0$.
If $H_k(U,U\setminus A)\otimes\IQ_p\ne0$, then by Lemma~\ref{dyer}, the group $H_k(U,U\backslash A)$ contains an element not divisible by $p$  and then $H_k(U,U\setminus A)\otimes\IZ_p\ne0$ by Lemma~\ref{dyer}. In both the cases, the Formula of Universal Coefficients implies that $H_k(U,U\setminus A;\IZ_p)\ne0$, which means that $A$ fails to be a $\IZ_p$-homological $Z_n$-set in $X$.
\smallskip

3. Assume that $A$ fails to be a $\IZ_p$-homological $Z_n$-set in $X$
 and find $k\le n$ and an open set $U\subset X$ with $H_k(U,U\setminus A;\IZ_p)\ne0$. The Formula of Universal Coefficients yields $H_k(U,U\setminus A)\otimes\IZ_p\ne0$ or $H_{k-1}(U,U\setminus A)*\IZ_p\ne0$. In the latter case $H_{k-1}(U,U\setminus A)$ contains an element of order $p$ and then $H_{k-1}(U,U\setminus A)*\IQ_p\ne0$. Applying the Formula of Universal Coefficients once more, we get $H_k(U,U\setminus A;\IQ_p)\ne0$, which means that $A$ fails to be a $\IQ_p$-homological $Z_k$-set.

Now assume that $H_k(U,U\setminus A)\otimes\IZ_p\ne0$. If $H_k(U,U\setminus A)$ contains an element of order $p$, then $H_k(U,U\setminus A)*\IQ_p\ne0$ and $H_{k+1}(U,U\setminus A;\IQ_p)\ne0$ by the Formula of Universal Coefficients. This means that $A$ fails to be a $\IQ_p$-homological $Z_{k+1}$-set in $X$. So it remains to consider the case when the group $H=H_k(U,U\setminus A)$ has trivial $p$-torsion part $\pTor(H)$.  Since $H\otimes \IZ_p\ne0$, the group $H=H/\pTor(H)$ contains an element $a$ not divisible by $p$, see Lemma~\ref{dyer}. Then $H_k(U,U\setminus A)\otimes\IQ_p=H\otimes\IQ_p\ne0$ according to Lemma~\ref{dyer}.
Now the Formula of Universal Coefficients implies that $H_k(U,U\setminus A;\IQ_p)\ne0$, which means that $A$ fails to be a $\IQ_p$-homological $Z_k$-set in $X$.

4. Assume that $A$ fails to be a $\Rp$-homological $Z_n$-set in $X$. Then for some $k\le n$ and an open set $U\subset X$ the homology group $H_k(U,U\setminus A;\Rp)=H_k(U,U\setminus A)\otimes\Rp$ is not trivial. By Lemma~\ref{dyer} the group $H_k(U,U\setminus A)$ contains
an element of infinite order or an element of order $p$. In the first case the group $H_k(U,U\setminus A;\IQ)=H_k(U,U\setminus A)\otimes \IQ$ is not trivial, which means that $A$ fails to be a $\IQ$-homological $Z_k$-set. In the second case the subgroup $H_k(U,U\setminus A)*\IQ_p\subset H_{k+1}(U,U\setminus A;\IQ_p)$ is not trivial, which means that $A$ fails to be a $\IQ_p$-homological $Z_{k+1}$-set in $X$.
\end{proof}

Combining Theorems~\ref{bockstein} and~\ref{bockstein2}, we get

\begin{corollary} Let $A$ be a closed subset of a topological space $X$.
\begin{enumerate}
\item If $A$ is a $G$-homological $Z_n$-set in $X$ for some coefficient group $G$, then $A$ is an $H$-homological $Z_n$ in $X$ for some divisible group $H\in\{\IQ,\IQ_p:p\in\Pi\}$.
\item $A$ is an $F$-homological $Z_n$-set in $X$ for a field $F$ if and only if $A$ is a $G$-homological $Z_n$-set in $X$ for the field $G\in\{\IQ,\IZ_p:p\in\Pi\}$ with $\{G\}=\sigma(F)$.
\end{enumerate}
\end{corollary}

\section{Multiplication Theorem for homotopical and homological $Z_n$-sets}\label{products}

In this section we discuss so-called Multiplication Theorems for homotopical and $R$-homological $Z_n$ sets with $R$ being a
{\em principal ideal domain}. The latter means that $R$ is a commutative ring with unit and without zero divisors in which each proper ideal is generated by a single element. A typical example of a principal ideal domain is the ring $\IZ$ of integers.

According to the K\"unneth Formula \cite[Th.5.3.10]{Spa}, for closed subsets $A\subset X$, $B\subset Y$ in topological spaces $X,Y$ and a principal ideal domain $R$ the relative homology group 
$H_n(X\times Y,X\times Y\setminus A\times B;R)$ is isomorphic to the direct sum of the $R$-modules 
$$
[H_*(X,X\setminus A;R)\otimes_R H_*(Y,Y\setminus B;R)]_n=
\oplus_{i+j=n}H_i(X,X\setminus A;R)\otimes_R H_j(Y,Y\setminus B;R)
$$and
$$
[H_*(X,X\setminus A;R)*_R H_*(Y,Y\setminus B;R)]_{n-1}=
\oplus_{i+j=n-1}H_i(X,X\setminus A;R)*_R H_j(Y,Y\setminus B;R).
$$
Here $G\otimes_R H$ and $G*_RH$ stand for the tensor and torsion products of $R$-modules $G,H$ over $R$. If $R=\IZ$, then we omit the subscript and write $G\otimes H$ and $G*H$. It is known that the torsion product over a field $F$ always is trivial. In this case, $$H_n(X\times Y,X\times Y\setminus A\times B;F)=[H_*(X,X\setminus A;F)\otimes_F H_*(Y,Y\setminus B;F)]_n.$$

With help of the K\"unneth Formula and Theorem~\ref{bockstein} we shall prove Multiplication Formulas for homological and homotopical $Z_n$-sets.

\begin{theorem}\label{multZn} Let $A\subset X$, $B\subset Y$ be closed subsets in Tychonov spaces $X,Y$, and $G$ be a coefficient group.
\begin{enumerate}
\item If $A$ is a $G$-homological $Z_n$-set in $X$ and $B$ is a $G$-homological $Z_m$-set in $Y$, then $A\times B$ is a $G$-homological $Z_{n+m}$-set in $X\times Y$. Moreover, if $\sigma(G)\subset\{\IQ,\IZ_p,\Rp:p\in\Pi\}$, then $A\times B$ is a $G$-homological $Z_{n+m+1}$-set in $X\times Y$.
\item If $A$ is a homotopical $Z_n$-set in $X$ and $B$ is a homotopical $Z_m$-set in $Y$, then $A\times B$ is a homotopical $Z_{n+m+1}$-set in $X\times Y$.
\end{enumerate}
\end{theorem}

\begin{proof}
1. In light of Theorem~\ref{bockstein}, it suffices to prove the first item only for a Bockstein group $G\in\{\IQ,\IZ_p,\Rp,\IQ_p:p\in\Pi\}$.
 Assume that $A$ is a $G$-homological $Z_n$-set in $X$ and $B$ is a $G$-homological $Z_m$-set in $Y$.

First, we prove the second part of the first item, assuming that $G\in\{\IQ,\IZ_p,\Rp:p\in\Pi\}$ is a principal ideal domain.
We have to check that $A\times B$ is a $G$-homological $Z_{n+m+1}$-set in $X\times Y$, which means that $H_k(W,W\setminus A\times B;G)=0$ for every $k\le n+m+1$ and every open set $W\subset X\times Y$. By Lemma~\ref{mult} it suffices to consider the case $W=U\times V$ for some open sets $U\subset X$ and $V\subset Y$.

By the K\"unneth Formula, the homology group 
$H_k(U\times V,  U\times V\,\setminus \,A\times B;G)$ is isomorphic to the direct sum of the groups
$$\bigoplus_{i+j=k} H_i(U,U\setminus A;G)\otimes_G H_j(V,V\setminus B;G)\mbox{ \ and \ }
\bigoplus_{i+j=k-1} H_i(U,U\setminus A;G)*_G H_j(V,V\setminus B;G).$$Observe that for every $i,j$ with $i+j\le n+m+1$ either $i\le n$ or $j\le m$. In the first case the group $H_i(U,U\setminus A;G)$ is trivial (because $A$ is a $G$-homological $Z_n$-set in $X$), in the second case the group $H_j(V,V\setminus B;G)=0$. Hence the above tensor and torsion products are trivial and so is the group $H_k(U\times V,U\times V\setminus A\times B;G)$.
\smallskip

Next, assume that $G=\IQ_p$ for a prime $p$. We need to prove that $A\times B$ is a $G$-homological $Z_{n+m}$-set in $X\times Y$. As in the preceding case this reduces to showing that for every $k\le n+m$ and open sets $U\subset X$, $V\subset Y$ the homology group $H_k(U\times V,U\times V\setminus A\times B;\IQ_p)$ is trivial.
By the Formula of Universal Coefficients, this group is isomorphic to the direct sum of the groups $H_k(U\times V,U\times V\setminus A\times B)\otimes\IQ_p$ and $H_{k-1}(U\times V,U\times V\setminus A\times B)*\IQ_p$. So it suffices to prove the triviality of these two groups.

The triviality of the group \mbox{$H_{k-1}(U\times V, U\times V\setminus A\times B)*\IQ_p$}  will follow as soon as we prove that the group $H_{k-1}(U\times V,U\times V\setminus A\times B)$ has no $p$-torsion. Assuming the converse and using the K\"unneth Formula, we conclude that either the group $$\bigoplus_{i+j=k-1}H_i(U,U\setminus A)\otimes H_j(V,V\setminus B)\mbox{ \ \ or the group \ }
\bigoplus_{i+j=k-2}H_i(U,U\setminus A)*H_j(V,V\setminus B)$$ contains an element of order $p$.

In the latter case there are $i,j$ with $i+j=k-2$ such that $H_i(U,U\setminus A)*H_j(V,V\setminus B)$ contains an element of order $p$. By Lemma~\ref{dyer} both the groups $H_i(U,U\setminus A)$ and $H_j(V,V\setminus B)$ contain elements of order $p$ and then $H_i(U,U\setminus A)*\IQ_p$ and $H_j(U,U\setminus B)*\IQ_p$ are not trivial and so are the homology group $H_{i+1}(U,U\setminus A;\IQ_p)$ and $H_{j+1}(V,V\setminus B;\IQ_p)$ by the Formula of Universal Coefficients. Since $A$ is a $\IQ_p$-homological $Z_n$-set in $X$ and $B$ is a $\IQ_p$-homological $Z_m$-set in $Y$, the non-triviality of the two latter homology groups implies that $i\ge n$ and $j\ge m$ and hence $k-2=i+j\ge n+m\ge k$, which is a contradiction.

Next, we consider the case when for some $i,j$ with $i+j=k-1$ the
group $H_i(U,U\setminus A)\otimes H_j(V,V\setminus B)$ contains
an element of order $p$. To simplify notation let
$H_i=H_i(U,U\setminus A)$, $H_j=H_j(V,V\setminus B)$.  Since
$H_i\otimes H_j$ contains an element of order $p$, we can apply
Lemma~\ref{dyer}(3) to conclude that either $H_i$ or $H_j$
contains an element of order $p$. Without loss of generality,
$H_i$ contains an element of order $p$. It follows from the
Formula of Universal Coefficients and the fact that $A$ is a
$\IQ_p$-homological $Z_n$-set in $X$ that $i\ge n$. Then
$j=k-1-i\le n+m-1-i< m$. Since $B$ is a $\IQ_p$-homological
$Z_m$-set in $Y$, the Formula of Universal Coefficients implies
that $H_j=H_j(V,V\setminus B)$ has no $p$-torsion. Taking into
account that $H_i$ has $p$-torsion and $H_i\otimes H_j$ contains
an element of order $p$, we can apply Lemma~\ref{dyer}(3) to
conclude that $H_j=H_j/\pTor(H_j)$ is not divisible by $p$ and
hence $H_j(V,V\setminus B)\otimes\IQ_p=H_j\otimes\IQ_p\ne0$ by
Lemma~\ref{dyer}. The Formula of Universal Coefficients implies
now that $H_j(V,V\setminus B;\IQ_p)\ne0$, which contradicts the
fact that $B$ is a $\IQ_p$-homological $Z_m$-set (because $j\le
m$). This completes the proof of the triviality of the group
$H_{k-1}(U\times V,U\times V\setminus A\times B)*\IQ_p$.

Next, we check the triviality of the group $H_{k}(U\times V,U\times V\setminus A\times B)\otimes \IQ_p.$ By the K\"unneth Formula, the group $H_k(U\times V,U\times V\setminus A\times B)$ is isomorphic to the direct sum of the groups
$$\bigoplus_{i+j=k}H_i(U,U\setminus A)\otimes H_j(V,V\setminus B)\mbox{ \ and \ }
\bigoplus_{i+j=k-1}H_i(U,U\setminus A)*H_j(V,V\setminus B).$$
The second sum is a torsion group and hence its tensor product with $\IQ_p$ is trivial. So, it suffices to prove that
$(H_i(U,U\setminus A)\otimes H_j(V,V\setminus B))\otimes\IQ_p=0$ for any $i,j$ with $i+j=k$. Since $k\le n+m$ either $i\le n$ or $j\le m$. If $i\le n$, we can use the fact that $A$ is a $\IQ_p$-homological $Z_n$-set in $X$ to conclude that the homology group $H_i(U,U\setminus A;\IQ_p)$ is trivial and so is the tensor product $H_i(U,U\setminus A)\otimes\IQ_p$ according to the Formula of Universal Coefficients. The associativity of the tensor product implies that  $\big(H_i(U,U\setminus A)\otimes H_j(V,V\setminus B)\big)\otimes\IQ_p$ is trivial  as well. If $j\le m$, then the triviality of the above tensor product follows from the $\IQ_p$-homological $Z_m$-set property of $B$.
\smallskip

2. Let $A$ be a homotopical $Z_n$-set in $X$ and $B$ is a homotopical $Z_m$-set in $Y$. It will be convenient to uniformize the notations and put $X_0=X$, $X_1=Y$, $A_0=A$, $A_1=B$, $k_0=n$, $k_1=m$ and $k=k_0+k_1+1$. So we need to prove that the product $A_0\times A_1$ is a homotopical $Z_k$-set in $X_0\times X_1$ provided $A_i$ is a homotopical $Z_{k_i}$-set in $X_i$ for $i\in\{0,1\}$.

Let $\U$ be an open cover of $X_0\times X_1$ and $f=(f_0,f_1):I^{k}\to X_0\times X_1$ be a map of the $k$-dimensional cube. Then the sets $f_i(I^{k})\subset X_i$, $i\in\{0,1\}$, are compact and so is their product. A standard compactness argument yields finite covers $\U_i$ of $f_i(I^{k})$ by open subsets of $X_i$ for $i\in\{0,1\}$ such that for any sets $U_0\in \U_0$, $U_1\in\U_1$ the product $U_0\times U_1$ lies in some set $U\in\U$.

 By Lemma~\ref{metric}, each space $X_i$ has a continuous pseudometric $\rho_i$  such that any 1-ball $B(x,1)=\{x'\in X_i:\rho_i(x,x')<1\}$ centered at a point $x\in f_i(I^{k})$ lies in some set $U\in\U_i$. Then the pseudometric $\rho=\max\{\rho_0,\rho_1\}$ $$\rho((x_0,x_1),(x'_0,y'_1))=\max\{\rho_0(x_0,x'_0),\rho_1(x_1,x'_1)\}$$
on $X_0\times X_1$ has a similar feature: any 1-ball $B(a,1)$ centered at a point $a\in f_0(I^k)\times f_1(I^k)$ lies in some set $U\in\U$.

Let $T$ be a triangulation of $I^{k}$ so fine that the image
$f(\sigma)$ of any simplex $\sigma\in T$ has $\rho$-diameter
$<1/3$. Let $K_0$ be the $k_0$-dimensional skeleton of the
triangulation $T$ and $K_1$ be the dual skeleton consisting of all
simplexes of the barycentric subdivision of $T$ that do not meet
the skeleton $K_0$. It is well-known (and easy to see) that $K_1$
has dimension $k_1=k-k_0-1$ and each point $z\in I^{k}$ lying in a
simplex $\sigma\in T$ can be written as
$z=(1-\lambda(z))z_0+\lambda(z)z_1$ for some points $z_i\in K_i$,
$i\in\{0,1\}$, and some real number $\lambda(z)\in[0,1]$. This
number is uniquely determined by the point $z$ and is equal to
zero iff $z\in K_0$ and equal to 1 iff $z\in K_1$. Moreover, the
point $z_i$ is uniquely determined by $z$ iff $z\notin K_{1-i}$
for $i\in\{0,1\}$. This means that the cube $I^k$ has the
structure of a subset of the join $K_0*K_1$. This structure
allows to write the cube $I^k$ as $I^k=K_{\le 1/2}\cup K_{\ge
1/2}$, where $K_{\le 1/2}=\{z\in I^k:\lambda(z)\le 1/2\}$ and
$K_{\ge 1/2}=\{z\in I^k:\lambda(z)\ge 1/2\}$.

Let $\ell:[0,1]\to[0,1]$ be the piecewise linear map determined by the
conditions $\ell(0)=0=\ell(1/2)$ and $\ell(1)=1$. Combined with
the joint structure $K_0*K_1$, the map $\ell$ induces two
piece-linear maps $h_i:I^k\to I^k$, $i\in\{0,1\}$, assigning to
each point $z=\lambda_0z_0+\lambda_1 z_1$ with $z_i\in K_i$,
$\lambda_0+\lambda_1=1$ the point
$h_i(z)=\lambda'_0z_0+\lambda_1'z_1$ where
$\lambda_i'=\ell(\lambda_i)$ and $\lambda'_{1-i}=1-\lambda'_i$.
The crucial property of the maps $h_i$ is that $h_0(K_{\le
1/2})\subset K_0$, $h_1(K_{\ge 1/2})\subset K_1$ and both $h_0$
and $h_1$ are $\mathcal S$-homotopic to the identity map of $I^k$
with respect to the cover $\mathcal S$ of $I^k$ by maximal
simplexes of the triangulation $T$.

Applying Theorem~\ref{torunczyk} to the homotopical $Z_{k_i}$-set $A_i\subset X_i$ find a map $g_i:K_i\to X_i\setminus A_i$, $1/6$-homotopic to the map $f_i|K_i:K_i\to X_i$ with respect to the pseudometric $\rho_i$. Since $K_i$ is a subcomplex of $I^{k}$, we may apply Borsuk Homotopy Extension Theorem (see \cite[1.D]{Spa}) and extend the map $g_i:K_i\to X_i$ to a map $\bar g_i:I^k\to X_i$, $1/6$-homotopic to $f_i$. Finally consider the map $\tilde f=(\tilde f_0,\tilde f_1):I^k\to X_0\times X_1$, where $\tilde f_i=\bar g_i\circ h_i$. We claim that $\tilde f$ is $\U$-homotopic to $f$ and $\tilde f(I^k)\cap (A_0\times A_1)=\emptyset$.

The $\mathcal S$-homotopy of the maps $h_i$ to the identity implies the $\bar g_i(\mathcal S)$-homotopy of $\bar g_i\circ h_i$ to $\bar g_i$. Now observe that for each simplex $\sigma$ of the triangulation $T$ we get
$$\diam(\bar g_i(\sigma))\le \diam (f_i(\sigma))+2\dist(f_i,\bar g_i)<\frac13+2\frac16=\frac23.$$
Consequently, $\tilde f_i$ is $2/3$-homotopic to $\bar g_i$. Since $\bar g_i$ is $1/6$-homotopic to $f_i$, we get that $\tilde f_i$ is 1-homotopic to $f_i$ and consequently, $\tilde f$ is $1$-homotopic to $f$. Now the choice of the pseudometric $\rho$ implies that $\tilde f$ is $\U$-homotopic to $f$.

So it remains to prove that $\tilde f(z)\notin A_1\times A_2$  for every point $z\in I^k$. Indeed, if $z\in K_{\le 1/2}$, then $h_0(z)\in K_0$ and $\tilde f_0(z)=\bar g_0\circ h_0(z)\in\bar g_0(K_0)\subset X_0\setminus A_0$. A similar argument yields $\tilde f_1(z)\notin A_1$ provided $z\in K_{\ge 1/2}$.
\end{proof}

\section{Functions $\zetaa(A,X)$ and $\zetaa^G(A,X)$}

It will be convenient to write Theorem~\ref{multZn} in terms of the functions
\smallskip

$\zetaa(A,X)=\sup\{n\in\w: A$ is a homotopical $Z_n$-set in $X\}$ and

$\zetaa^G(A,X)=\sup\{n\in\w: A$ is a $G$-homological $Z_n$-set in $X\}$
\smallskip

\noindent defined for a closed subset $A$ of a space $X$ and a coefficient
group $G$. In this definition we put $\sup\emptyset=-1$. So
$\zetaa^G(A,X)=-1$ iff $A$ is not a $G$-homotopical $Z_0$-set in
$X$. For a point $x\in X$ of a topological space $X$ we write
$\zetaa^G(x,X)$ instead of $\zetaa^G(\{x\},X)$.

Rewriting Theorems~\ref{Zsets},  \ref{Z2+hZn=Zn}, \ref{indP1}, \ref{bockstein},
and \ref{bockstein2} in the terms of the functions $\zetaa(A,X)$ and $\zetaa^G(A,X)$ we obtain

\begin{theorem}\label{zeta} Let $A$ be a closed subset of a topological space $X$, and $G$ be a coefficient group. Then
\begin{enumerate}
\item $\zetaa(A,X)\le\zetaa^{\IZ}(A,X)\le\zetaa^G(A,X)=\min_{H\in\sigma(G)}\zetaa^H(A,X)$;
\item $\zetaa(A,X)=\zetaa^{\IZ}(A,X)$ if $X$ is an $\LC[1]$-space and $\zetaa(A,X)\ge2$;
\item $\zetaa^G(A,X)+\trt(A)\ge\min_{a\in A}\zetaa^G(a,X)$;
\item $\min\{\zetaa^{\IQ}(A,X),\zetaa^{\IQ_p}(A,X)-1\}\le
\zetaa^{\Rp}(A,X)$ and  $\zetaa^{\Rp}(A,X)\le\min\{\zetaa^{\IQ}(A,X),\zetaa^{\IZ_p}(A,X)\}$;
\item $\zetaa^{\IZ_p}(A,X)\le\zetaa^{\IQ_p}(A,X)\le\zetaa^{\IZ_p}(A,X)+1$.
\end{enumerate}
\end{theorem}

Also, Theorem~\ref{multZn} can be rewritten in the form of
Multiplication Formulas.

\begin{theorem}[Multiplication Formulas]\label{multform} Let $A,B$ be closed subsets in Tychonov spaces $X,Y$, and $G$ be a coefficient group. Then
\begin{enumerate}
\item $\zetaa^G(A\times B,X\times Y)\ge \zetaa^G(A,X)+\zetaa^G(B,Y)$;
\item $\zetaa^G(A\times B,X\times Y)\ge \zetaa^G(A,X)+\zetaa^G(B,Y)+1$ if   $\sigma(G)\subset\{\IQ,\IZ_p,\Rp:p\in\Pi\}$;
\item $\zetaa(A\times B,X\times Y)\ge \zetaa(A,X)+\zetaa(B,Y)+1$.
\end{enumerate}
\end{theorem}

\section{Division and $k$-Root Formulas for homological $Z_n$-sets}\label{div}

It turns out that inequalities in the Multiplication Formulas~\ref{multform} can be partly reversed, which leads to so-called Division and $k$-Root Formulas.
We start with Division Formulas for Bockstein coefficient groups.

\begin{lemma}\label{divbock} Let $A\subset X$, $B\subset Y$, $C\subset Z$ be closed subsets in topological spaces $X,Y,Z$, and $p$ be a prime number. Then
\begin{enumerate}
\item $\zetaa^F(A\times B,X\times Y)=\zetaa^F(A,X)+\zetaa^F(B,X)+1$ for a field $F$;
\item $\zetaa^{\IQ_p}(A\times B,X\times Y)\le \zetaa^{\IQ_p}(A,X)+\zetaa^{\IZ_p}(B,Y)+2$;
\item $\zetaa^{\Rp}(A,X){+}1\ge \min\{\zetaa^{\Rp}(A{\times} B,X{\times} Y){-}\zetaa^\IQ(B,Y),\zetaa^{\Rp}(A{\times} C,  X{\times} Z){-}\zetaa^{\IQ_p}(C,Z)\}$ if\newline   $\max\{\zetaa^{\IQ}(B,Y),\zetaa^{\IQ_p}(C,Z)\}<\infty$.
\end{enumerate}
\end{lemma}

\begin{proof} 1. Let $F$ be a field. By the Multiplication Theorem~\ref{multZn}(1), $\zetaa^F(A\times B,X\times Y)\ge \zetaa^F(A,X)+\zetaa^F(B,Y)+1$. So it remains to prove the reverse inequality, which is trivial if one of the numbers $n=\zetaa^F(A,X)$ or $m=\zetaa^F(B,Y)$ is infinite. So assume that $n,m<\infty$ and find
open sets $U\subset X$ and $V\subset Y$ such that the homology groups $H_{n+1}(U,U\setminus A;F)$ and $H_{m+1}(V,V\setminus B;F)$ are not trivial.
Their tensor product $H_{n+1}(U,U\setminus A;F)\otimes_{F} H_{m+1}(V,V\setminus B;F)$ over the field $F$ is not trivial as well. Now the K\"unneth Formula implies that the group $H_{n+m+2}(U\times V,U\times V\setminus A\times B;F)$ is not trivial, which means that $A\times B$ is not an $F$-homological $Z_{n+m+2}$-set in $X\times Y$ and hence $\zetaa^F(A\times B,X\times Y)\le n+m+1=\zetaa^F(A,X)+\zetaa^F(B,Y)+1$.
\smallskip

2. The second item follows from the first one and Theorem~\ref{zeta}(5): $$\zetaa^{\IQ_p}(A{\times} B,X{\times} Y)\le \zetaa^{\IZ_p}(A{\times} B,X{\times} Y){+}1=
\zetaa^{\IZ_p}(A,X){+}\zetaa^{\IZ_p}(B,Y){+}2\le\zetaa^{\IQ_p}(A,X){+}\zetaa^{\IZ_p}(B,Y){+}2.$$

3. To prove the third item, let $n=\zetaa^{\Rp}(A,X)$ and find an
open set $U\subset X$ and  $i\le n+1$ with
$H_{i}(U,U\setminus A;\Rp)\ne0$. The Formula of Universal
Coefficients yields $H_i(U,U\setminus A)\otimes\Rp\ne0$. By
Lemma~\ref{dyer}, the group $H_i(U,U\setminus A)$ contains an
element of infinite order or of order $p$.

In the first case, use the fact that $B$ is not a $\IQ$-homological $Z_{m+1}$-set in $Y$ for $m=\zetaa^{\IQ}(B,Y)<\infty$ to find an open set $V\subset Y$ with $H_{m+1}(V,V\setminus B;\IQ)\ne0$. By the Formula of Universal Coefficients, $H_{m+1}(V,V\setminus B)\otimes\IQ\ne0$ and hence $H_{m+1}(V,V\setminus B)$ contains an element of infinite order and so does the tensor product
 $H_i(U,U\setminus A)\otimes H_{m+1}(V,V\setminus B)$, which lies in the homology group $H_{i+m+1}(U\times V,U\times V\setminus A\times B)$ according to the K\"unneth Formula. Then the tensor product $$H_{i+m+1}(U\times V,U\times V\setminus A\times B)\otimes\Rp=H_{i+m+1}(U\times V,U\times V\setminus A\times B;\Rp)$$is not trivial. This means that $A\times B$ fails to be a
$\Rp$-homological $Z_{i+m+1}$-set in $X\times Y$ and thus $\zetaa^{\Rp}(A\times B,X\times Y)\le i+m\le 1+n+m$. Consequently, $$\zetaa^{\Rp}(A,X)+1=n+1\ge \zetaa^{\Rp}(A\times B,X\times Y)-\zetaa^{\IQ}(B,Y).$$

Next, assume that $H_i(U,U\setminus A)$ contains an element of
order $p$ and use the fact that the set $C$ fails to be a
$\IQ_p$-homological $Z_{m+1}$-set in $Z$ for the number
$m=\zetaa^{\IQ_p}(C,Z)<\infty$ to find an open set $V\subset Z$
with $H_{m+1}(V,V\setminus C;\IQ_p)\ne0$. Then either
$H_{m+1}(V,V\setminus C)\otimes\IQ_p\ne0$ or $H_{m}(V,V\setminus
C)*\IQ_p\ne0$. In the latter case $H_{m}(V,V\setminus C)$
contains an element of order $p$ and so does the torsion product
$H_i(U,U\setminus A)*H_{m}(V,V\setminus C)$ which lies in
$H_{i+m+1}(U\times V,U\times V\setminus A\times C)$ by the
K\"unneth Formula. Applying Lemma~\ref{dyer}, we see that
$$H_{i+m+1}(U\times V,U\times V\setminus A\times
C)\otimes\Rp=H_{i+m+1}(U\times V,U\times V\setminus A\times
C;\Rp)$$ is not trivial. This means that $A\times C$ is not a
$\Rp$-homological $Z_{i+m+1}$-set in $X\times Z$.

Finally assume that  $H_{m+1}(V,V\setminus C)\otimes\IQ_p\ne0$. By Lemma~\ref{dyer}, the group $H_{m+1}(V,V\setminus C)/\pTor(H_{m+1}(V,V\setminus C))$ is not divisible by $p$ and hence the tensor product $H_i(U,U\setminus A)\otimes H_{m+1}(V,V\setminus C)$ contains an element of order $p$.
By the K\"unneth Formula, the latter tensor product lies in $H_{i+m+1}(U\times V,U\times V\setminus A\times C)$. Applying Lemma~\ref{dyer} again, we obtain that the tensor product $$H_{i+m+1}(U\times V,U\times V\setminus A\times C)\otimes\Rp=H_{i+m+1}(U\times V,U\times V\setminus A\otimes C;\Rp)$$ is not trivial. In both the cases $A\times C$ is not a $\Rp$-homological $Z_{i+m+1}$-set in $X\times Z$ and hence $$\zetaa^{\Rp}(A\times C,X\times Z)\le i+m\le n+1+\zetaa^{\IQ_p}(C,Z).$$ Then $\zetaa^{\Rp}(A,X)+1=n+1\ge \zetaa^{\Rp}(A\times C,X\times Z)-\zetaa^{\IQ_p}(C,Z)$.
\end{proof}

Now we can establish Division Formulas in the general case. First, to each coefficient group $G$ assign two families $d(G),\varphi(G)\subset\{\IQ,\IZ_p,\IQ_p:p\in\Pi\}$ as follows. Put
\begin{itemize}
\item $d(\IQ)=\varphi(\IQ)=\{\IQ\}$,
\item $d(\IZ_p)=\varphi(\IZ_p)=\{\IZ_p\}$,
\item $d(\IQ_p)=\varphi(\IQ_p)=\{\IZ_p\}$,
\item $d(\Rp)=\{\IQ,\IQ_p\}$, $\varphi(\Rp)=\{\IQ,\IZ_p\}$
\end{itemize}and let
$$d(G)=\bigcup_{H\in\sigma(G)}d(H)\mbox{\;\; and\;\;}\varphi(G)=\bigcup_{H\in\sigma(G)}\varphi(H).$$
In particular, $d(\IZ)=\{\IQ,\IQ_p:p\in\Pi\}$ is the family of divisible Bockstein groups and $\varphi(\IZ)=\{\IQ,\IZ_p:p\in\Pi\}$ is the family of Bockstein fields.

%We define a coefficient group $G$ to be {\em ring-like} if $\sigma(G)\subset\{\IQ,\IZ_p,\Rp:p\in\Pi\}$.

\begin{theorem}[Division Formulas]\label{divisionZ} Let $A\subset X$, $B\subset Y$ be closed subsets in topological spaces and $G$ be a coefficient group. Then
\begin{enumerate}
\item $\zetaa^G(A\times B,X\times Y)\le 2+\zetaa^G(A,X)+\sup_{F\in\varphi(G)}\zetaa^{F}(B,Y)$;
\item $\zetaa^G(A\times B,X\times Y)\le 1+\zetaa^G(A,X)+\sup_{H\in d(G)}\zetaa^H(B,Y)$ provided $\sigma(G)\subset\{\IQ,\IZ_p,\Rp:p\in\Pi\}$.
\end{enumerate}
\end{theorem}

\begin{proof} 1. The inequality  $\zetaa^G(A\times B,X\times Y)\le 2+\zetaa^G(A,X)+\sup_{F\in\varphi(G)}\zetaa^{F}(B,Y)$ is trivial if $n=\zetaa^G(A,X)$ is infinite. So assume that $n$ is finite and using Theorem~\ref{zeta}(1), find a Bockstein group $D\in\sigma(G)$ with $\zetaa^D(A,X)=n$. Theorem~\ref{bockstein} implies that $\zetaa^G(A\times B,X\times Y)\le \zetaa^D(A\times B,X\times Y)$. If $D$ is a field, then $\{D\}=\varphi(D)\subset\varphi (G)$ and $$\zetaa^G(A\times B,X\times Y)\le\zetaa^D(A\times B,X\times Y)=1+\zetaa^D(A,X)+\zetaa^D(B,Y)\le 1+\zetaa^G(A,X)+\sup_{H\in \varphi(G)}\zetaa^H(B,Y)$$ by Lemma~\ref{divbock}.

If $D=\Rp$ for some $p$, then
$\varphi(\Rp)=\{\IQ,\IZ_p\}\subset\varphi(G)$ and by Lemma~\ref{divbock} we get
\begin{multline*}
\zetaa^G(A\times B,X\times Y)\le\zetaa^{\Rp}(A\times B,X\times Y)\le 1+\zetaa^{R_p}(A,X)+\max\{\zetaa^{\IQ}(B,Y),\zetaa^{\IQ_p}(B,Y)\}\le\\
\le 1+\zetaa^{\Rp}(A,X)+\max\{\zetaa^{\IQ}(B,Y),\zetaa^{\IZ_p}(B,Y)+1\}\le\\ \le 2+\zetaa^{D}(A,X)+\max\{\zetaa^{\IQ}(B,Y),\zetaa^{\IZ_p}(B,Y)\} 
\le 2+\zetaa^G(A,X)+\sup_{F\in\varphi(G)}\zetaa^F(B,Y).
\end{multline*}

If $D=\IQ_p$ for some $p$, then
$$\zetaa^G(A{\times} B,X{\times} Y)\le \zetaa^{\IQ_p}(A{\times} B,X{\times} Y)\le
 2{+}\zetaa^{\IQ_p}(A,X){+}\zetaa^{\IZ_p}(B,Y)\le 2{+}\zetaa^{G}(A,X){+}\sup_{F\in\varphi(G)}\zetaa^F(B,Y)$$according to Lemma~\ref{divbock}(2).
\smallskip

2. If $\sigma(G)\subset\{\IQ,\IZ_p,\Rp:p\in\Pi\}$ then the last case in the preceding item is excluded and we can repeat the preceding argument to prove the inequality
$$\zetaa^G(A\times B,X\times Y)\le 1+\zetaa^G(A,X)+\sup_{H\in d(G)}\zetaa^H(B,Y).$$
\end{proof}

\begin{theorem}\label{divhomZ} Let $A\subset X$, $B\subset Y$ be closed nowhere dense subsets in Tychonov spaces $X,Y$. Then
\begin{enumerate}
\item $1+\zetaa(A,X)+\zetaa(B,Y)\le \zetaa(A\times B,X\times Y)\le\zetaa^{\IZ}(A\times B,X\times Y)$;
\item $\zetaa(A\times B,X\times Y)=\zetaa^{\IZ}(A\times B,X\times Y)$ provided $X,Y$ are $\LC[1]$-spaces.
\end{enumerate}
\end{theorem}

\begin{proof} The first item follows from Theorem~\ref{multZn}(2) and \ref{zeta}(1).

To prove the second item assume that $X,Y$ are $\LC[1]$-spaces, consider three cases.

(i) $\zetaa^{\IZ}(A,X)=\zetaa^{\IZ}(B,X)=0$. In this case
 $\zetaa(A,X)=\zetaa^{\IZ}(A,X)=\zetaa^{\IQ}(A,X)$ and $\zetaa(B,Y)=\zetaa^{\IZ}(B,Y)=\zetaa^{\IQ}(B,Y)$. So 
$\zetaa(A,X)+\zetaa(B,X)+1=\zetaa^{\IQ}(A,X)+\zetaa^{\IQ}(B,Y)+1=\zetaa^{\IQ}(A\times B,X\times Y)\ge \zetaa(A\times B,X\times Y)\ge \zetaa(A,X)+\zetaa(B,Y)+1$
by Lemma~\ref{divbock}(1) and Theorem~\ref{multZn}(2).

(ii) $\zetaa^{\IZ}(A,X)>0$. In this case $A$ is a homological $Z_1$-set in $X$ and a homotopical $Z_1$-set in the $\LC[1]$-space $X$ by Theorem~\ref{Zsets}(6,3). The set $B$, being nowhere dense in the $\LC[1]$-space $Y$, is a homotopical $Z_0$-set in $Y$. Then $A\times B$, being the product of a homotopical $Z_1$-set and a homotopical $Z_0$-set, is a homotopical $Z_2$-set in $X\times Y$ by Theorem~\ref{multZn}(2). By Theorem~\ref{Z2+hZn=Zn}, $A\times B$ is a  homotopical $Z_n$-set in $X\times Y$ if and only if it is a homological $Z_n$-set in $X\times Y$, which  implies the desired equality $\zetaa(A\times B,X\times Y)=\zetaa^{\IZ}(A\times B,X\times Y)$.

(iii) The case $\zetaa^{\IZ}(B,Y)>0$ can be considered by analogy.
\end{proof}

Theorem~\ref{divhomZ} implies

\begin{corollary}\label{prodcor} A subset $A$ of a Tychonov $\LC[1]$-space $X$ is a homological $Z_n$-set in $X$ if and only if $A\times\{0\}$ is a homotopical $Z_{n+1}$-set in $X\times [-1,1]$.
\end{corollary}

Next we turn to the $k$-Root Theorem.

\begin{theorem}[$k$-Root Theorem]\label{kRootZ} Let $A$ be a closed subset in a topological space $X$, $k\in\IN$, and $G$ be a coefficient group. Then
\begin{enumerate}
\item $k\cdot\zetaa^G(A,X)\le \zetaa^G(A^k,X^k)\le  k\cdot\zetaa^G(A,X)+2k-2$.
\item $k\cdot\zetaa^G(A,X)+k-1\le \zetaa^G(A^k,X^k)$ provided  $\sigma(G)\subset\{\IQ,\IZ_p,\Rp:p\in\Pi\}$;
\item $\zetaa^G(A^k,X^k)\le k\cdot\zetaa^G(A,X)+k$,
provided $\sigma(G)\subset\{\IQ,\IZ_p,\IQ_p:p\in\Pi\}$;
\item $\zetaa^G(A^k,X^k)=k\cdot\zetaa^G(A,X)+k-1$, provided  $\sigma(G)\subset\{\IQ,\IZ_p:p\in\Pi\}$;
\end{enumerate}
\end{theorem}

\begin{proof} The items of this theorem will be proved in the following order: 2,4,3,1. There is nothing to prove if $k=1$. So we assume that $k\ge 2$.

2. The second item follows by induction from the Multiplication Formula~\ref{multform}(2).

4. Assume that $\sigma(G)\subset\{\IQ,\IZ_p:p\in\Pi\}$. By Theorem~\ref{zeta}(1), there is a Bockstein group $F\in\sigma(G)$ with $\zetaa^F(A,X)=\zetaa^G(A,X)$.
Since $F$ is a field, we may apply Lemma~\ref{divbock}(1) $k-1$ times and get the equality $\zetaa^F(A^k,X^k)=k\cdot\zetaa^F(A,X)+k-1$. Then $$k\cdot \zetaa^G(A,X)+k-1=k\cdot \zetaa^F(A,X)+k-1= \zetaa^F(A^k,X^k)\ge \zetaa^G(A^k,X^k)\ge k\cdot \zetaa^G(A,X)+k-1$$ (the last inequality follows from the preceding item)
which yields the equality from the fourth item of the theorem.
\smallskip

3. Assume that $\sigma(G)\subset\{\IQ,\IZ_p,\IQ_p:p\in\Pi\}$ and using Theorem~\ref{zeta}(1), find a Bockstein group $H\in\sigma(G)$ with $\zetaa^G(A,X)=\zetaa^H(A,X)$.
 For this group we also get $\zetaa^G(A^k,X^k)\le\zetaa^H(A^k,X^k)$ according to Theorem~\ref{zeta}(1).
If  $H\in\{\IQ,\IZ_p:p\in\Pi\}$, then $$\zetaa^G(A^k,X^k)\le\zetaa^H(A^k,X^k)=k\cdot \zetaa^H(A,X)+k-1=k\cdot \zetaa^G(A,X)+k-1$$ by the preceding case.

So it remains to consider the case $H=\IQ_p$ for a prime $p$. Applying Theorem~\ref{zeta}(5) and the second item of this theorem, we get
$$
\zetaa^G(A^k,X^k)\le \zetaa^{\IQ_p}(A^k,X^k)\le \zetaa^{\IZ_p}(A^k,X^k)+1=\\
k\cdot\zetaa^{\IZ_p}(A,X)+k\le k\cdot \zetaa^{\IQ_p}(A,X)+k =k\cdot \zetaa^G(A,X)+k.
$$

1. The inequality $k\cdot\zetaa^G(A,X)\le \zetaa^G(A^k,X^k)$ follows by induction from Theorem~\ref{multform}(1). To prove the inequality $\zetaa^G(A^k,X^k)\le   k\cdot \zetaa^G(A,X)+2k-2$, apply Theorem~\ref{zeta}(1) to find a group $H\in\sigma(G)$ such that $\zetaa^H(A,X)=\zetaa^G(A,X)$. If $H\subset\{\IQ,\IZ_p,\IQ_p:p\in\Pi\}$, then $$
\zetaa^G(A^k,X^k)\le\zetaa^H(A^k,X^k)\le k\cdot \zetaa^H(A,X)+k=
k\cdot \zetaa^G(A,X)+k\le \zetaa^G(A,X)+2k-2.$$

So it remains to consider the case of $H=\Rp$ for a prime $p$ and
prove that $\zetaa^G(A^k,X^k)\le  k\cdot \zetaa^G(A,X)+2k-2$.
Assuming the converse, we conclude that $A^k$ is a
$G$-homological $Z_{kn+k-1}$-set in $X^k$ for
$n=\zetaa^G(A,X)+1=\zetaa^{\Rp}(A,X)+1$. It follows from the
definition of $\zetaa^{\Rp}(A,X)=n-1$ that the homology group
$H_{n}(U,U\setminus A;\Rp)=H_{n}(U,U\setminus A)\otimes\Rp$ is
not trivial for some open set $U\subset X$. Applying
Lemma~\ref{dyer}, we get that  $H_{n}(U,U\setminus A)$ contains
an element $a$ that has either infinite order or order $p$.

If $a$ has infinite order then the tensor product $H_n(U,U\setminus A)\otimes H_n(U,U\setminus A)$ contains an element of infinite order and so does group $H_{2n}(U^2,U^2\setminus A^2)$ according to the K\"unneth Formula.
Now, we can show by induction that the group $H_{kn}(U^k,U^k\setminus A^k)$ contains an element of infinite order, which implies that
$H_{kn}(U^k,U^k\setminus A^k)\otimes\Rp=H_{kn}(U^k,U^k\setminus A^k;\Rp)\ne0$ and hence $A^k$ fails to be a $\Rp$-homological $Z_{kn}$-set in $X^k$. Then we get a contradiction: $$\zetaa^G(A^k,X^k)\le\zetaa^{\Rp}(A^k,X^k)<kn\le kn+k-1\le \zetaa^G(A^k,X^k).$$
If $a$ is of order $p$, then the torsion product \mbox{$H_n(U,U\backslash A)*H_n(U,U\backslash A)$} contains an element of order $p$ and so does group $H_{2n+1}(U^2,U^2\setminus A^2)$ according to the K\"unneth Formula.  Proceeding by induction, we can show that the homology group $H_{i(n+1)-1}(U^i,U^i\setminus A^i)$ contains an element of order $p$. For $i=k$, the group  $H_{kn+k-1}(U^k,U^k\setminus A^k)$ contains an element of order $p$ and hence $H_{kn+k-1}(U^k,U^k\setminus A^k)\otimes \Rp=H_{kn+k-1}(U^k,U^k\setminus A^k;\Rp)$ is not trivial by Lemma~\ref{dyer}. This means that $A^k$ is not a $\Rp$-homological $Z_{kn+k-1}$-set in $X$ and thus $$\zetaa^G(A^k,X^k)\le\zetaa^{\Rp}(A^k,X^k)\le kn+k-2<\zetaa^G(A^k,X^k),$$
which is a contradiction.
\end{proof}

\section{$Z_n$-points}\label{Zpoints}

A point $x$ of a space $X$ is defined to be a {\em $Z_n$-point} if
its singleton $\{x\}$ is a $Z_n$-set in $X$. By analogy we define
homotopical and $G$-homological $Z_n$-points. It is clear that all
results proved in the preceding sections for $Z_n$-sets concern
also $Z_n$-points. For $Z_n$-points there are however some
simplifications. In particular, the Excision Property for Singular
Homology Theory (see Theorem 4 in \cite[4.6]{Spa}) allows us to
characterize homological $Z_n$-points as follows.

\begin{proposition} For a point $x$ of a
regular topological space $X$ and a coefficient group $G$
the following conditions are equivalent:
\begin{enumerate}
\item $x$ is a $G$-homological $Z_n$-point in $X$;
\item $H_k(X,X\setminus\{x\};G)=0$ for all $k<n+1$;
\item there is an open neighborhood $U\subset X$ of $x$ such that 
$H_k(U,U\setminus\{x\};G)=0$ for all $k<n+1$.
\end{enumerate}
\end{proposition}

In this section, given a topological space $X$ and a coefficient group $G$ we shall study the Borel complexity of the sets:
\begin{itemize}
\item $\Z_n(X)$ of all homotopical $Z_n$-points in $X$, and
\item $\Z_n^G(X)$ of all $G$-homological $Z_n$-points in $X$.
\end{itemize}

\begin{theorem}\label{Gdelta}  Let $X$ be a metrizable separable space and $G$ be a coefficient group.
 Then
\begin{enumerate}
\item the set  of $Z_n$-points in $X$ is a $G_\delta$-set in $X$;
\item the set $\Z_n(X)$ of homotopical $Z_n$-points is a $G_\delta$-set in $X$ provided $X$ is an $\LC[n]$-space;
\item the set of $G$-homological $Z_n$-points is a $G_\delta$-set in
$X$ if $|H_k(U;G)|\le\aleph_0$ for all open sets $U\subset X$ and all
$k<n+1$;
\item the set $\Z^G_n(X)$ of $G$-homological $Z_n$-points is a $G_\delta$-set in
$X$ provided is an $\lc[n]$-space;
\item the set of homotopical $Z_n$-points is a $G_\delta$-set in
$X$ if $X$ is an $\LC[2]$-space and $|H_k(U)|\le\aleph_0$ for all
open subsets $U\subset X$ and all $k<n+1$.
\end{enumerate}
\end{theorem}

\begin{proof} 1. Let $\rho$ be a metric on $X$ generating the topology of $X$.
Since the space $X$ is metrizable and separable, so is the
function space $C(I^n,X)$ endowed with the sup-metric
$$\tilde\rho(f,g)=\sup_{z\in I^n}\rho(f(z),g(z)).$$ So, we may fix a
countable dense set $\{f_i\}_{i=1}^\infty$ in $C(I^n,X)$. Observe
that a closed subset $A\subset X$ fails to be a $Z_n$-set if and
only if there are a function $f\in C(I^k,X)$ and $\e>0$ such that
for each function $g\in C(I^n,X)$ with $\tilde \rho(f,g)<\e$ the
image $g(I^n)$ meets the set $A$. The density of $\{f_i\}$ implies
that $\tilde \rho(f,f_i)<\e/2$ for some $i\in\w$. Consequently,
$g(I^k)\cap A\ne\emptyset$ for all $g\in C(I^n,X)$ with $\tilde
\rho(g,f_i)<\e/2$.

Now for each $i,k\in\IN$ consider the set $$F_{i,k}=\{x\in
X:\forall g\in C(I^n,X)\; \;\;\tilde \rho(g,f_i)<1/k\Ra x\in
g(I^n)\}.$$ It is easy to see that the set $F_{i,k}$ is closed in
$X$. Moreover the preceding discussion implies that
$F=\bigcup_{i,k=1}^\infty F_{i,k}$ is the set of all points of $X$
that fail to be $Z_n$-points in $X$. Then its complement is a
$G_\delta$-set coinciding with the set of all $Z_n$-points in $X$.
\smallskip

2. If $X$ is an $\LC[n]$-space, then each $Z_n$-point is a homotopical $Z_n$-point in $X$ and consequently the set $\Z_n(X)$ of homotopical $Z_n$-points coincides with the $G_\delta$-set $X\setminus F$ of all $Z_n$-points.
\smallskip

3. Assume that the homology groups $H_k(U;G)$ are countable for
all $k<n+1$ and all open sets $U\subset X$. Let $\mathcal
B=\{U_i:i\in\w\}$ be a countable base of the topology for $X$ with
$U_0=\emptyset$. For every $i\in\w$ and $k<n+1$ use the
countability of the groups $H_k(X\setminus\overline{U}_i;G)$ to
find a countable sequence $(\alpha_{i,j})_{j\in\w}$ of cycles in
$X\setminus\overline{U}_i$ whose representatives $[\alpha_{i,j}]$
exhaust all non-zero elements of the groups $H_k(X\setminus\bar
U_i;G)$, $k<n+1$.

For every $j\in\w$ consider the open set $W_{0,j}=\{x\in X:
\alpha_{0,j}$ is homologous to some cycle in $X\setminus\{x\}\}$.
Also for every $i\in\IN$ and $j\in\w$ consider the open set
$W_{i,j}=\{x\in U_i:$ if $\alpha_{i,j}$ is null-homological in
$X$, then it is null-homological in $X\setminus\{x\}\}$. It
remains to prove that the $G_\delta$-set
$Z=\bigcap_{i,j\in\w}W_{i,j}$ coincides with the set $\Z_n^G(X)$ of all
$G$-homological $Z_n$-points in $X$. It is clear that $Z$ contains
all $G$-homological $Z_n$-points of $X$. Now take any point $x\in
Z$. Assuming that  $x$ is not a $G$-homological $Z_n$-point, find
$k<n+1$ such that $H_k(X,X\setminus\{x\};G)\ne0$. The exact
sequence $$H_{k}(X\setminus\{x\};G)\to H_{k}(X;G)\to
H_k(X,X\setminus\{x\};G)\to H_{k-1}(X\setminus\{x\};G)\to
H_{k-1}(X;G)$$ of the pair $(X,X\setminus\{x\})$ now implies that
for $m=k$ or $m=k-1$ the inclusion homomorphism
$i:H_m(X\setminus\{x\};G)\to H_m(X;G)$ fails to be an isomorphism.

If $i$ is not onto, then for some $j\in\w$ the element
$\alpha_{0,j}$ is homologous to no cycle in $X\setminus\{x\}$.
This means that $x\notin W_{0,j}$ and thus $x\notin Z$ which is a
contradiction.

If $i$ is not injective, then there is a $k$-cycle $\alpha$ in
$X\setminus\{x\}$ which is homologous to zero in $X$ but not in
$X\setminus\{x\}$. Since $\alpha$ has compact support, there is a
basic neighborhood $U_i$ of $x$ such that $\alpha$ is supported by
the set $X\setminus \overline{U}_i$. Find $j\in\w$ such that the
cycle $\alpha_{i,j}$ is homologous to $\alpha$ in $X\setminus
\overline{U}_i$. It follows that $\alpha_{i,j}$ is homologous to
zero in $X$ but not in $X\setminus \{x\}$. Then $x\notin W_{i,j}$
and hence $x\notin Z$. This contradiction completes the proof of
the third item.
\medskip

4. Assuming that $X$ is an $\lc[n]$-space we shall show that the set $\Z^G_n(X)$ is of type $G_\delta$ in $X$. Since $\Z_n^G(X)=\bigcap_{H\in\sigma(G)}\Z_n^H(X)$, it suffices to check that $\Z_n^H(X)$ is a $G_\delta$-set in $X$ for every countable group $H$. This will follows from the preceding item as soon as we show that for every open set $U\subset X$ the homology groups $H_i(U;H)$, $i\le n$, are at most countable.
In its turn, this will follow from the Formula of Universal Coefficients as soon as we check that the homology groups $H_i(U)$, $i\le n$, are at most countable.

Fix a countable family $\mathcal K$ of compact polyhedra containing a topological copy of each compact polyhedron. For each polyhedron $K\in\K$ fix a countable dense set $\mathcal F_K$ in the function space $C(K,U)$. Note that for each polyhedron $K\in\K$ the homology group $H_*(K)$ is finitely generated and hence at most countable.

For every homology element $\alpha\in H_i(U)$ with $i\le n$ there is a continuous map $f:K\to U$ of a compact polyhedron $K\in\K$ such that $\alpha\in f_*(H_i(K))$. Moreover, according to Lemma~\ref{lcn} we can assume that $f\in\mathcal F_K$. Then the homology group$$H_i(U)=\bigcup_{K\in\K}\bigcup_{f\in\mathcal F_K}f_*(H_i(K))$$is countable, being the countable union of finitely generated groups.
\smallskip

5. The fifth item follows from items (2), (4), and the characterization of homotopical $Z_n$-sets given by Theorem~\ref{Z2+hZn=Zn}.
\end{proof}

\section{On spaces whose all points are $Z_n$-points}\label{Zn}

In this section we introduce three classes of Tychonov spaces related to $Z_n$-points:
\begin{itemize}
\item $\Z_n$ the class of spaces $X$ with $X=\Z_n(X)$,
\item $\Z_n^G$ the class of spaces $X$ with $X=\Z_n^G(X)$;
\item $\cup_G\Z_n^G$ the union of classes $\Z_n^G$ over all coefficient groups.
\end{itemize}

For example, $\IR^{n+1}$ belongs to all of these classes while $\IR^n$ belongs to none of them. The classes $\Z_n$ play an important role in studying the general position properties from \cite{BV}.

By $\LC[n]$ (resp. $\lc[n]$) we shall denote the class of metrizable  $\LC[n]$-spaces (resp. $\lc[n]$-spaces).

The following corollary describes the relation between the introduced classes and can be easily derived from Theorems~\ref{Zsets}, \ref{Z2+hZn=Zn},  \ref{bockstein}, and \ref{bockstein2}.

\begin{corollary}\label{Zn-classes} Let $n\in\w\cup\{\infty\}$ and $G$ be a coefficient group. Then
\begin{enumerate}
\item $\Z_n\subset\Z_n^{\IZ}\subset\Z_n^G=\bigcap_{H\in\sigma(G)}\Z_n^H$;
\item $\Z_0=\Z_0^{\IZ}=\Z_0^G$;
\item $\LC[1]\cap \Z_1^{G}\subset\Z_1$;
\item $\LC[1]\cap\Z_2\cap\Z_n^{\IZ}\subset\Z_n$;
\item $\cup_G\Z^G_n=\Z^{\IQ}_n\cup\bigcup_{p\in\Pi}\Z_n^{\IQ_p}$;
\item $\Z_n^{\IZ}=\bigcap_{p\in\Pi}\Z_n^{\Rp}$;
\item $\Z_n^{\Rp}\subset\Z_n^{\IQ}\cap\Z_n^{\IZ_p}$, $\Z_n^{\IZ_p}\subset\Z_n^{\IQ_p}\subset\IZ_{n-1}^{\IZ_p}$, $\Z_n^{\IQ}\cap \Z_{n+1}^{\IQ_p}\subset\Z_n^{\Rp}$ for every prime number $p$.
\end{enumerate}
\end{corollary}

For a better visual presentation of our subsequent results, let us introduce the following operations on subclasses $\mathcal A,\mathcal B\subset\Top$ of the class $\Top$ of topological spaces:
$$\begin{gathered}
\mathcal A\times\mathcal B=\{A\times B:A\in\mathcal A,\; B\in\mathcal B\},\\
\frac{\mathcal A}{\mathcal B}=\{X\in\Top:\exists B\in\mathcal B\mbox{ with } X\times B\in\mathcal A\},\\
\mathcal A^k=\{A^k:A\in\mathcal A\}\mbox{  and }\sqrt[k]{\mathcal A}=\{A\in\Top:A^k\in\mathcal A\}.
\end{gathered}
$$

Multiplication Formulas~\ref{multform} imply the following three Multiplication Formulas for the classes $\Z_n$ and $\Z_n^G$.
\smallskip

\begin{theorem}[Multiplication Formulas] Let $n,m\in\w\cup\infty$, $X,Y$ be Tychonov spaces.
\begin{enumerate}
\item If $X\in\Z_n$ and $Y\in\Z_m$, then $X\times Y\in\Z_{n+m+1}$:
\smallskip

\frame{\phantom{$\Big|$}
$\Z_n\times\Z_m\subset \Z_{n+m+1}$\phantom{$\Big|$}}
\smallskip

\item If  $X\in\Z^G_n$, $Y\in\Z^G_m$ for a coefficient group $G$, then $X\times Y\in\Z^G_{n+m}$:
\smallskip

\frame{\phantom{$\Big|$}
$\Z^G_n\times\Z^G_m\subset \Z^G_{n+m}$\phantom{$\Big|$}}

\item If  $X\in\Z^R_n$, $Y\in\Z^R_m$ for a coefficient group $R$ with
$\sigma(R)\subset\{\IQ,\IZ_p,\Rp:p\in\Pi\}$, then $X\times
Y\in\Z^R_{n+m+1}$:
\smallskip

\frame{\phantom{$\Big|$}
$\Z^R_n\times\Z^R_m\subset \Z^R_{n+m+1}$\phantom{$\Big|$}}
\end{enumerate}
\end{theorem}

The Multiplication Formulas for the classes $\Z_n^G$ can be reversed.

\begin{theorem}[Division Formulas] Let $n,m\in\w\cup\infty$, $X,Y$ be Tychonov spaces.
\begin{enumerate}
\item If $X\times Y\in\Z_{n+m+1}^{G}$ for a coefficient group $G$, then either $X\in\Z_n^{G}$ or $Y\in\bigcup_{F\in\varphi(G)}\Z_{m}^{F}$. This can be written as
\smallskip

\frame{\phantom{$\Big|^{\Big|}$}
$\dfrac{\Z^{G}_{n+m+1}}{\Top\setminus \Z^G_n}\subset\bigcup_{F\in\varphi(G)}\Z_{m}^{F}
$\phantom{$\Big|_{\Big|}$}}
\smallskip

\item If $X\times Y\in\Z_{n+m}^{R}$ for a coefficient group $R$ with $\sigma(R)\subset\{\IQ,\IZ_p,\Rp:p\in\Pi\}$, then either $X\in\Z_n^{R}$ or $Y\in\bigcup_{H\in d(R)}\Z_{m}^{H}$. This can be written as
\smallskip

\frame{\phantom{$\Big|^{\Big|}$}
$\dfrac{\Z^{R}_{n+m}}{\Top\setminus \Z^R_n}\subset\bigcup_{H\in d(R)}\Z_{m}^{H}
$\phantom{$\Big|_{\Big|}$}}
\end{enumerate}
\end{theorem}

\begin{proof} 1. Assume that $X\notin\Z_n^G$ and $Y\notin\bigcup_{F\in\varphi(G)}\Z^F_m$. Then there is a point $x\in X$ with $\zetaa^G(x,X)<n$ and for every field $F\in\varphi(G)$ there is a point $y_F\in Y$ with $\zetaa^F(y_F,Y)<m$. By Theorem~\ref{bockstein}, $\zetaa^G(x,X)=\zetaa^{H}(x,X)$ for some group $H\in\sigma(G)$.
If $H$ is a field, then $H\in\varphi(H)\subset\varphi(G)$ and by Lemma~\ref{divbock}, $\zetaa^H((x,y_H),X\times Y)=\zetaa^H(x,X)+\zetaa^H(y_H,Y)+1\le (n-1)+(m-1)+1=n+m-1$, which means that $(x,y_H)$ is not an $H$-homological $Z_{n+m}$-set in $X\times Y$ and thus $X\times Y\notin\Z_{n+m}^H\supset \Z_{n+m}^G$.

If $H=\Rp$ for some $p$, then $\{\IQ,\IQ_p\}=d(\Rp)$ and $\IZ_p\in \varphi(G)$. By Lemma~\ref{divbock}(3),
$$\zetaa^H(x,X)+1\ge \min\{\zetaa^H((x,y_{\IQ}),X\times Y)-\zetaa^{\IQ}(y_{\IQ},Y),  \zetaa^H((x,y_{\IZ_p}),X\times Y)-\zetaa^{\IQ_p}(y_{\IZ_p},Y)\}.$$ Therefore, either  $\zetaa^H((x,y_{\IQ}),X\times Y)\le 1+\zetaa^H(x,X)+\zetaa^{\IQ}(y_{\IQ},Y)\le (n-1)+(m-1)+1$ or $\zetaa^H((x,y_{\IZ_p}),X\times Y)\le 1+ \zetaa^H(x,X)+\zetaa^{\IQ_p}(y_{\IZ_p},Y)\le    1+(n-1)+\zetaa^{\IZ_p}(y_{\IZ_p},Y)+1=n+m.$
In both cases, $X\times Y\notin \Z^H_{n+m+1}$.

If $H=\IQ_p$, then $\IZ_p\in\varphi(G)$. By Lemma~\ref{divbock},
$$\zetaa^H((x,y_{\IZ_p}),X\times Y)\le \zetaa^H(x,X)+\zetaa^{\IZ_p}(y_{\IZ_p},Y)+2\le (n-1)+(m-1)+2=n+m,$$
which implies that $X\times Y\notin\Z^H_{n+m+1}$.
\smallskip

2. Assume that $X\notin\Z_n^R$ and $Y\notin\bigcup_{F\in d(R)}\Z^F_m$ for a group $R$ with $\sigma(R)\subset\{\IQ,\IZ_p,\Rp:p\in\Pi\}$. Then there is a point $x\in X$ with $\zetaa^R(x,X)<n$ and for every group $H\in d(R)$ there is a point $y_H\in Y$ with $\zetaa^H(y_H,Y)<m$.
  By Theorem~\ref{bockstein}, $\zetaa^R(x,X)=\zetaa^{H}(x,X)$ for some group $H\in\sigma(R)\subset\{\IQ,\IZ_p,\Rp:p\in\Pi\}$.  Theorem~\ref{bockstein} implies that $\Z^R_{n+m}\subset\Z_{n+m}^H$.

If $H$ is a field, then repeating the reasoning from the preceding item, we can prove that $$\zetaa^H((x,y_H),X\times Y)=\\ \zetaa^H(x,X)+\zetaa^H(y_H,Y)+1\le (n-1)+(m-1)+1=n+m-1$$ which means that $X\times Y\notin\Z^H_{n+m}$.

If $H=\Rp$ for some $p$, then $\{\IQ,\IQ_p\}=d(\Rp)\subset d(G)$. By Lemma~\ref{divbock}(3),
$$\zetaa^H(x,X)+1\ge\\ \min\{\zetaa^H((x,y_{\IQ}),X{\times} Y)-\zetaa^{\IQ}(y_{\IQ},Y),\zetaa^H((x,y_{\IZ_p}),X{\times} Y)-\zetaa^{\IQ_p}(y_{\IQ_p},Y)\}.$$ Therefore, either $$\zetaa^H((x,y_{\IQ}),X\times Y)\le 1+\zetaa^H(x,X)+\zetaa^{\IQ}(y_{\IQ},Y)\le (n-1)+(m-1)+1=n+m-1\mbox{ or}$$
$$\zetaa^H((x,y_{\IZ_p}),X\times Y)\le 1+ \zetaa^H(x,X)+\zetaa^{\IQ_p}(y_{\IQ_p},Y)\le 1+(n-1)+(m-1)=n+m-1.$$ 
In both cases, $X\times Y\notin \Z^H_{n+m}\supset\Z^R_{n+m}$.
\end{proof}

Finally, we prove  $k$-Root Formulas for the classes $\Z_n^G$.

\begin{theorem}[$k$-Root Formulas] Let $n\in\w\cup\infty$, $k\in\IN$, $X$ be a topological space, and $G$ be a coefficient group.
\begin{enumerate}
\item  If $X^k\in\Z^{G}_{kn+k-1}$, then $X\in\Z_n^{G}$:
\smallskip

\frame{\phantom{$\Bigg|$}
$\sqrt[k]{\Z^{G}_{nk+k-1}}\subset\Z^{G}_n$ \phantom{$\Bigg|$}}
\smallskip

\item If $\sigma(G)\subset\{\IQ,\IZ_p,\IQ_p:p\in\Pi\}$ and $X^k\in\Z^{G}_{kn+1}$, then $X\in\Z_n^{G}$:
\smallskip

\frame{\phantom{$\Bigg|$}
$\sigma(G)\subset\{\IQ,\IZ_p,\IQ_p:p\in\Pi\}\;\Rightarrow\;\sqrt[k]{\Z^G_{nk+1}}\subset\Z^{G}_n$ \phantom{$\Bigg|$}}
\smallskip

\item If $\sigma(G)\subset\{\IQ,\IZ_p:p\in\Pi\}$ and $X^k\in\Z^{G}_{kn}$, then $X\in\Z_n^{G}$:
\smallskip

\frame{\phantom{$\Bigg|$}
$\sigma(G)\subset\{\IQ,\IZ_p:p\in\Pi\}\;\Rightarrow\;\sqrt[k]{\Z^G_{nk}}=\Z^{G}_n$ \phantom{$\Bigg|$}}
\smallskip
\end{enumerate}
\end{theorem}

\begin{proof} Assume that $X\notin\Z^G_n$ and find a point $x\in X$ with  $\zetaa^G(\{x\},X)<n$.

1. Applying Theorem~\ref{kRootZ}(1), we get $\zetaa^G(\{x\}^k,X^k)<k\cdot \zetaa^G(x,X)+2k-1\le k(n-1)+2k-1=kn+k-1$  and hence $X^k\notin\Z^G_{kn+k-1}$.
\smallskip

2. If $\sigma(G)\subset\{\IQ,\IZ_p,\IQ_p:p\in\Pi\}$, then we can apply Theorem~\ref{kRootZ}(3) to conclude that $\zetaa^G(\{x\}^k,X^k)\le k\cdot \zetaa^G(x,X)+k\le k(n-1)+k=kn$ and $X^k\notin\Z^G_{kn+1}$.
\smallskip

3. If $\sigma(G)\subset\{\IQ,\IZ_p:p\in\Pi\}$, then we can apply Theorem~\ref{kRootZ}(4) to conclude that $\zetaa^G(\{x\}^k,X^k)=k\cdot \zetaa^G(x,X)+k-1\le k(n-1)+k-1=kn-1$ and hence $X^k\notin\Z^G_{kn}$.
\end{proof}

\section{On spaces containing a dense set of (homological) $Z_n$-points}\label{DZn}

In this section we consider classes of Tychonov spaces containing dense sets of (homological) $Z_n$-points. More precisely, let
\begin{itemize}
\item $\DZ_n$ be the class of spaces $X$ with dense set $\Z_n(X)$ of  homotopical $Z_n$-points in $X$;
\item $\DZ_n^G$ be the class of spaces $X$ with dense set $\Z_n^G(X)$ of  $G$-homological $Z_n$-points;
\item $\cup_G\DZ_n^G$ be the union of classes $\DZ_n^G$ over all coefficient groups.
\end{itemize}

The classes $\DZ_n$ play an important role in the paper \cite{BV} devoted to general position properties.

It is clear that $\Z_n\subset\DZ_n$. On the other hand, a dendrite $D$ with dense set of end-points belongs to $\DZ_\infty$ but not to $\Z_1$.
By $\Baire$ we shall denote the class of metrizable separable Baire spaces.

The following proposition describes relation between the introduced classes.

\begin{proposition}\label{Zn-classes-p} Let $n\in\w\cup\{\infty\}$ and $G$ be a coefficient group. Then
\begin{enumerate}
\item $\DZ_n\subset\DZ_n^{\IZ}\subset\DZ_n^G\subset\bigcap_{H\in\sigma(G)}\DZ^H_n$;
\item $\Baire\cap \lc[n]\cap\bigcap_{H\in\sigma(G)}\DZ^H_n\subset\DZ^G_n$;
\item $\DZ_0=\DZ_0^G$;
\item $\LC[1]\cap \DZ_1^{G}\subset\DZ_1$;
\item $\LC[1]\cap\Z_2\cap\DZ_n^{\IZ}\subset\DZ_n$;
\item $\LC[n]\cap\Baire\cap\DZ_2\cap\DZ_n^{\IZ}\subset\DZ_n$.
\end{enumerate}
\end{proposition}

\begin{proof} 1. The first item follows from Theorem~\ref{Zsets}(4), Proposition~\ref{Z->G}, and Theorem~\ref{bockstein}.
\smallskip

2. If $X\in\bigcap_{H\in\sigma(G)}\DZ^H_n$ is a metrizable separable Baire $\lc[n]$-space, then for every group $H\in\sigma(G)$ the set $\Z_n^H(X)$ of $H$-homological $Z_n$-points is a dense $G_\delta$-set in X according to Theorem~\ref{Gdelta}(4). By Theorem~\ref{bockstein}, the intersection $\bigcap_{H\in\sigma(G)}\Z_n^H(X)$ coincides with $\Z^G_n(X)$ and is dense in  $X$, being the countable intersection of dense $G_\delta$-sets in the Baire space $X$. Hence $X\in\DZ_n^G$.
\smallskip

3-5. The items (3)--(5) follow immediately from Theorems~\ref{Zsets} and \ref{Z2+hZn=Zn}.

6. Assume that $X\in\DZ_2\cap\DZ^{\IZ}$ is a metrizable separable Baire $\LC[n]$-space. By Theorem~\ref{Gdelta}, the sets $\Z_2(X)$ and $\Z_n^{\IZ}(X)$ are dense $G_\delta$ in $X$. Since $X$ is Baire, the intersection $\Z_2(X)\cap \Z^{\IZ}_n(X)$ is dense $G_\delta$ in $X$. By Theorem~\ref{Z2+hZn=Zn}, the latter intersection consists of homotopical $Z_n$-points. So $X\in\DZ_n$.
\end{proof}

Multiplication Formulas for the classes $\DZ_n$ and $\DZ^G_n$ follow immediately from Multiplication Theorem~\ref{multZn} for homotopical and homological $Z_n$-sets.

\begin{theorem}[Multiplication Formulas]\label{multDZ} Let $n,m\in\w\cup\infty$ and $X,Y$ be Tychonov spaces.
\begin{enumerate}
\item If $X\in\DZ_n$ and $Y\in\DZ_m$, then $X\times Y\in\DZ_{n+m+1}$\textup{:}
\smallskip

\frame{\phantom{$\Big|$}
$\DZ_n\times\DZ_m\subset \DZ_{n+m+1}$\phantom{$\Big|$}}
\smallskip

\item If $X\in\DZ^G_n$ and $Y\in\DZ^G_m$ for a coefficient group $G$, then $X\times Y\in\DZ^G_{n+m}$\textup{:}
\smallskip

\frame{\phantom{$\Bigg|$}
$\DZ^G_n\times\DZ^G_m\subset \DZ^G_{n+m}$\phantom{$\Big|$}}

\item If $X\in\DZ^R_n$ and $Y\in\DZ^R_m$ for a coefficient group $R$ with $\sigma(R)\subset\{\IQ,\IZ_p,\Rp:p\in\Pi\}$, then $X\times Y\in\DZ^R_{n+m+1}$\textup{:}
\smallskip

\frame{\phantom{$\Bigg|$}
$\DZ^R_n\times\DZ^R_m\subset \DZ^R_{n+m+1}$\phantom{$\Big|$}}

\end{enumerate}
\end{theorem}

Division and $k$Root Formulas for the classes $\DZ^G_n$ look as follows. For a class $\mathsf C$ of topological spaces by $\exists_\circ\kern-1pt\mathsf C$ we denote the class of topological spaces $X$ containing a non-empty open subspace $U\in\mathsf C$.

\begin{theorem}[Division Formulas]\label{division3} Let $X,Y$ be topological spaces.
\begin{enumerate}
\item If $X\times Y\in\DZ_{n+m}^F$ for some field $F$, then either $X\in\DZ_n^F$ or $Y\in\DZ_m^F$\textup{:}
\smallskip

\frame{\phantom{$\Big|^{\Big|}$}
$\dfrac{\DZ^F_{n+m}}{\Top\setminus \DZ^F_m}=\DZ^F_n$ \phantom{$\Big|_{\Big|}$}}
\smallskip

\item Assume that $X\times Y\in\DZ_{n+m}^{R}$ for a coefficient group $R$ with $\sigma(R)\subset\{\IQ,\IZ_p,\Rp:p\in\Pi\}$  and either $\sigma(R)$ is finite or $X$ is a metrizable separable Baire $\lc[n]$-space. Then either $X\in\DZ_n^{R}$ or else some non-empty open set $U\subset Y$ belongs to the class $\bigcup_{H\in d(R)}\DZ_{m}^{H}$\textup{:}
\smallskip

\frame{\phantom{$\Big|^{\Big|}$}
$\dfrac{\DZ^{R}_{n+m} }
{\Baire\cap\lc[n]\setminus\DZ^R_n}\subset \exists_\circ\kern-2pt\bigcup_{H\in d(R)}\DZ^{H}_{m}$\phantom{$\Big|_{\Big|}$}}

\item Assume that $X\times Y\in\DZ_{n+m+1}^{G}$ for a coefficient group $G$ and either $\sigma(G)$ is finite or $X$ is a metrizable separable Baire $\lc[n]$-space. Then either $X\in\DZ_n^{G}$ or else some non-empty open set $U\subset Y$ belongs to the class $\bigcup_{F\in\varphi(G)}\DZ^F_{m}$\textup{:}
\smallskip

\frame{\phantom{$\Big|^{\Big|}$}
$\dfrac{\DZ^{G}_{n+m+1} }
{\Baire\cap\lc[n]\setminus\DZ^G_n}\subset \exists_\circ\kern-2pt\bigcup_{F\in\varphi(G)}\DZ^F_{m}$\phantom{$\Big|_{\Big|}$}}
\end{enumerate}
\end{theorem}

\begin{proof} 1. Assume that $X\times Y\in\DZ^F_{n+m}$ for some field $F$, but $X\notin\Z^F_n$ and $Y\notin\DZ_m^F$. Let $U$ be the interior of the set $X\setminus\Z^F_n(X)$ and $V$ be the interior of the set $Y\setminus \Z^F_m(Y)$. The product $U\times V$ is a non-empty open subset of $X\times Y$ and thus contains some $F$-homological $Z_{n+m}$-point $(x,y)$ in $X\times Y$. By Lemma~\ref{divbock}(1) either $x\in \Z^F_n(X)$ or $y\in \Z^F_m(Y)$. Both cases are not possible by the choice of $U,V$. This contradiction shows that $X\in\DZ^F_n$ or $Y\in\DZ_m^F$.
\smallskip

2. Assume that $X\times Y\in\DZ_{n+m}^R$ for a group $R$ with $\sigma(R)\subset\{\IQ,\IZ_p,\Rp:p\in\Pi\}$  but $X\notin\DZ^{R}_n$. The latter means that $X$ contains a non-empty open set $W\subset X$ disjoint with the set $\Z_n^{R}(X)$ of $R$-homological $Z_n$-points of $X$.
It is necessary to find an open set $U\in \bigcup_{H\in d(R)}\DZ_m^H$.
Assume conversely that no such a set $U$ exists.
This means that for every $H\in d(G)$ the set $\Z_m^H(Y)$ is nowhere dense in $Y$.

If $\sigma(R)$ is finite, then so is the set $d(R)$ and we can
find a non-empty open set $U\subset X$ disjoint with
$\bigcup_{H\in d(R)}\Z_m^H(Y)$.

By Theorem~\ref{divisionZ}(2), no point $(x,y)\in W\times U$ is an $R$-homological $Z_{n+m}$-point in $X\times Y$. But this contradicts the density of the set $\Z_{n+m}^{R}(X\times Y)$ in $X\times Y$.

Next, we consider the case when $\sigma(R)$ is infinite and $X$ is a metrizable separable Baire $\lc[n]$-space. By Theorem~\ref{Gdelta}, for every group $H\in\sigma(R)$ the set $\Z_n^H(X)$ is of type $G_\delta$ in $X$ and by Theorem~\ref{bockstein}, $\Z_n^G(X)=\bigcap_{H\in\sigma(R)}\Z_n^H(X)$. Assuming that $\Z_n^R(X)$ is not dense in the Baire space $X$, we conclude that for some group $H\in\sigma(R)$ the set $\Z_n^H(X)$ is not dense in $X$.  Since $\sigma(H)=\{H\}$ is finite, we may apply the preceding case to conclude that some nonempty open set $U\subset Y$ belongs to the class $$\bigcup_{D\in d(H)}\DZ^D_m\subset \bigcup_{D\in d(G)}\DZ^D_m.$$

3. Assume that $X\times Y\in\DZ_{n+m+1}^G$ for a coefficient group $G$ but $X\notin\DZ^{G}_n$. The latter means that $X$ contains a non-empty open set $W\subset X$ disjoint with the set $\Z_n^{G}(X)$ of $G$-homological $Z_n$-points of $X$.
It is necessary to find an open set $U\in \bigcup_{F\in\varphi(G)}\DZ_{m}^F$.
Assume conversely that no such a set $U$ exists.
This means that for every $F\in \varphi(G)$ the set $\Z_{m}^F(Y)$ is nowhere dense in $Y$.

If $\sigma(G)$ is finite, then so is the set $\varphi(G)$, and we
can find a non-empty open set $U\subset X$ disjoint with
$\bigcup_{F\in\varphi(G)}\Z_{m}^F(Y)$. By
Theorem~\ref{divisionZ}(1), no point $(x,y)\in W\times U$ is a
$G$-homological $Z_{n+m+1}$-point in $X\times Y$. But this
contradicts the density of the set $\Z_{n+m+1}^{G}(X\times Y)$ in
$X\times Y$.

The case of infinite $\sigma(G)$ can be reduced to the previous case by the same argument as in the item 2.\end{proof}

\begin{theorem}[$k$-Root Formulas] Let $n\in\w\cup\{\infty\}$, $k\in\IN$,  $X$ be a topological space, and $G$ be a coefficient group.
\begin{enumerate}
\item $X\in\DZ^F_n$ for some field $F$ if and only if $X^k\in\DZ_{nk}^{F}$:
\smallskip

\frame{\phantom{$\Bigg|$}
$\sqrt[k]{\DZ^F_{nk}}= \DZ^{F}_n$\phantom{ \;$\Bigg|$}}
\smallskip

\item If $X$ is a metrizable separable Baire $\lc[n]$-space and  $X^k\in\DZ_{nk+k}^{G}$, then $X\in\DZ_n^{G}$:
\smallskip

\frame{\phantom{$\Bigg|$}
$\Baire\cap\lc[n]\cap\sqrt[k]{\DZ^{G}_{nk+k}}\subset \DZ^{G}_n$\phantom{$\Bigg|$\;}}

\item If $X$ is a metrizable separable Baire $\lc[nk+k-1]$-space and $X^k\in\DZ_{nk+k-1}^{G}$, then $X\in \DZ^G_n$:
\smallskip

\frame{\phantom{$\Bigg|$}
$\Baire\cap\lc[kn+k-1]\cap\sqrt[k]{\DZ^{G}_{nk+k-1}}= \DZ^{G}_n$\phantom{$\Bigg|$\;}}
\end{enumerate}
\end{theorem}

\begin{proof} 1. The first item can be derived by induction from Theorems~\ref{multDZ}(2) and \ref{division3}(1).
\smallskip

2. Assume that $X$ is a metrizable separable Baire $\lc[n]$-space with $X^k\in\DZ^G_{kn+k}$. It follows from Theorem~\ref{Gdelta} that for every group $H\in\sigma(G)$ the set $\Z^H_n(X)$ is of type $G_\delta$ in $X$. Since $\Z_n^G(X)=\bigcap_{H\in\sigma(G)}\Z_n^H(X)$, the density of $\Z_n^G(X)$ in the Baire space $X$ will follow as soon as we prove the density of $\Z_n^H(X)$ in $X$ for each group $H\in\sigma(G)$.

If $H$ is a field, then the inclusion $X^k\in\DZ^G_{kn+k}\subset\DZ^H_{kn+k}$ combined with the first item implies that $X\in\DZ^H_{n+1}\subset\DZ^H_n$, which means that the set $\Z_n^H(X)$ is dense in $X$.

If $H=\Rp$ for some $p$, then Theorem~\ref{bockstein2}(1) implies that $X^k\in\DZ^{\Rp}_{nk+k}\subset\DZ^{\IQ}_{nk+k}\cap\DZ^{\IZ_p}_{nk+k}$ and the first item yields $X\in\DZ^{\IQ}_{n+1}\cap\DZ^{\IZ_p}_{n+1}$. Applying
Theorem~\ref{bockstein2}(2,4), we get $$X\in \DZ^{\IQ}_{n+1}\cap \DZ^{\IZ_p}_{n+1}\subset\DZ^{\IQ}_n\cap\DZ^{\IQ_p}_{n+1}\subset\DZ^{\Rp}_n=\DZ^H_n.$$

If $H=\IQ_p$ for some $p$, then Theorem~\ref{bockstein2}(3) implies that $X^k\in\DZ^{\IQ_p}_{nk+k}\subset \DZ^{\IZ_p}_{nk+k-1}\subset \DZ^{\IZ_p}_{nk}$ and the first item combined with Theorem~\ref{bockstein2}(2) implies $X\in\DZ^{\IZ_p}_n\subset\DZ^{\IQ_p}_n=\DZ^H_n$.
\smallskip

3. Assume that $X$ is a metrizable separable Baire $\lc[nk+k-1]$-space and $X^k\in\DZ^G_{nk+k-1}$.  Arguing as in the preceding case, we can reduce the problem to the case $G=\Rp$ for some prime number $p$.
By Theorem~\ref{Gdelta}, $\Z_{kn+k-1}^{\Rp}(X^k)$ is a dense $G_\delta$-set in $X^k$. Assuming that $\Z_n^{\Rp}(X)$ is not dense in $X$, find a non-empty open set $U\subset X$ disjoint with $\Z^{\Rp}_n(X)$.
Since $\DZ^{\Rp}_{nk+k-1}\subset\DZ^{\IQ}_{nk+k-1}\subset\DZ^{\IQ}_{nk}$, the first item implies that $X\in\DZ^{\IQ}_n$. Then $D=U\cap \DZ^{\IQ}_n(X)$ is a dense $G_\delta$-set in $U$ by Theorem~\ref{Gdelta} and hence $D^k$ is a dense $G_\delta$ set in $U^k$. Since $X^k$ is a Baire space, there is a point $\vec x\in D^k\cap \Z_{nk+k-1}^{\Rp}(X^k)$ which can be written as $\vec x=(x_1,\dots,x_k)$.

Each point $x_i$ is a $\IQ$-homological but not $\Rp$-homological $Z_n$-point in $X$. This means that for some $n_i\le n$ the group $H_{n_i}(X,X\setminus\{x_i\};\Rp)=H_{n_i}(X,X\setminus \{x_i\})\otimes\Rp$ is not trivial. Since $x_i$ is a $\IQ$-homological $Z_n$-set in $X$, the group $H_{n_i}(X,X\setminus\{x_i\})$ cannot contain an element of infinite order. Since $H_{n_i}(X,X\setminus\{x_i\})\otimes\Rp\ne0$, the group $H_{n_i}(X,X\setminus \{x_i\})$ contains an element of order $p$.

Let $m_i=i-1+\sum_{j=1}^i n_j$ for $i\le k$.
The torsion product $H_{n_1}(X,X\setminus\{x_1\})*H_{n_2}(X,X\setminus\{x_2\})$ contains an element of order $p$ and so does the group $H_{m_2}(X^2,X^2\setminus\{(x_1,x_2)\})$ by the K\"unneth Formula.
Now by induction we can show that for every $i\le k$ the homology group $H_{m_i}(X^i,X^i\setminus\{(x_1,\dots,x_i)\})$ contains an element of order $p$. For $i=k$, we get that the group $H_{m_k}(X^k,X^k\setminus\{\vec x\})$ contains an element of order $p$ and hence the tensor product
$$H_{m_k}(X^k,X^k\setminus\{\vec x\})\otimes\Rp=H_{m_k}(X^k,X^k\setminus\{\vec x\};\Rp)$$ is not trivial which is not possible because $m_k\le nk+k-1$ and $\vec x$ is a $\Rp$-homological $Z_{nk+k-1}$-point in $X^k$.
\end{proof}

\begin{remark} An example of a space in which $Z_\infty$-points
form a dense $G_\delta$-set is a dendrite $D$ with dense set of
end-points. Being large in sense of Baire category, the set of
$Z_\infty$-point of $D$ is small in geometric sense: it is locally
$\infty$-negligible in $D$. On the other hand, W.Kuperberg
\cite{Kuper} has constructed a finite polyhedron $P$ in which the
set of all $Z_\infty$-points fails to be locally
$\infty$-negligible. Yet, according to  I.Namioka \cite{Na}, the
set $Z$ of $Z_\infty$-points in a finite-dimensional
$\LC[\infty]$-space $X$ is small in a homological sense:
$H_k(U,U\setminus Z)=0$ for any open set $U\subset X$ and any
$k\in\w$. By its spirit this Namioka's result is near to our
Theorem~\ref{trans} asserting that a closed $\trt$-dimensional subspace $A\subset X$
 consisting of $G$-homological
$Z_\infty$-points of $X$ is a $G$-homological $Z_\infty$-set in
$X$.
\end{remark}

\section{$Z_n$-points and dimension}\label{sect:dim}

In this section we study the dimension properties of spaces whose all points are homological $Z_n$-points.
First we note a simple corollary of Theorem~\ref{indP1}.

\begin{theorem}\label{ind} If a \textup{(}separable metrizable\textup{)} topological space  $X\in\cup_G\Z_n^G$, then $\mathrm{trind}(X)\ge\trt(X)\ge 1+n$ \textup{(}and hence $\dim(X)\ge 1+n$\textup{)}.
\end{theorem}

A similar lower bound holds also for the cohomological and extension dimensions.
Given a space $X$ and a CW-complex $L$ we write $\edim
X\le L$ if any map $f:A\to L$ defined on a closed subset $A\subset
X$ extends to a map $\bar f:X\to L$. The Extension Dimension $\edim$
generalizes both the usual covering dimension $\dim $ and the
cohomological dimension $\dim_G$ because:
\begin{itemize}
\item $\dim X\le n$ iff
$\edim X\le S^n$ and
\item $\dim_GX\le n$ iff $\edim X\le K(G,n)$,
\end{itemize}
where $K(G,n)$ is an Eilenberg-MacLane complex, i.e., a CW-complex
$K$ with a unique non-trivial homotopy group $\pi_n(K)=G$.

The main (and technically most difficult) result of this section is

\begin{theorem}\label{cohdim} If $X\in\Z_n^{\IZ}$ is a locally compact $\lc[n]$-space, then $\dim_GX\ge n+1$ for any non-trivial abelian group $G$.
\end{theorem}

\begin{proof} Assume that $\dim_GX=n<\infty$ for some abelian group $G$. By
Theorem 2 of \cite{Kuz} (or Theorem 1.8 of \cite{Dr}), the space
$X$ contains a point $x$ having an open neighborhood $U\subset X$
with compact closure such that for any smaller neighborhood
$V\subset U$ of $x$ the homomorphism in the relative \v Cech
cohomology groups $i_{V,U}:\check H^n(X,X\setminus \bar V;G)\to \check
H^n(X,X\setminus \bar U;G)$, induced by the inclusion $(X,X\setminus
\bar U)\subset (X,X\setminus \bar V)$ is non-trivial.
The complete regularity of the locally compact space $X$ allows us to find a compact $G_\delta$-set $K_1\subset U$ containing $x$ in its interior.

It is well-known (see \cite[VI.\S9]{Spa}) that in paracompact $\lc[n]$-spaces \v Cech cohomology coincide with singular cohomology.
%Since $X$ is an $\lc[n]$-space, the \v Cech cohomologies in the preceding assertion can be replaced with singular cohomologies, see \cite[VI.\S9]{Spa}.
Singular cohomology relates to singular homology via the
following exact sequence, see \cite[\S3.1]{Hat}: {\small $$
0\to\mathrm{Ext}(H_{n-1}(X,A),G)\to H^n(X,A;G)\to \mathrm
{Hom}(H_n(X,A),G)\to 0.$$} This sequence will be applied to the
pairs $$(X,X\setminus K_1)\subset (X,X\setminus K_2)\subset(X,X\setminus K_3)$$ where $K_3\subset K_2\subset K_1$ are compact $G_\delta$-neighborhoods of $x$ so small that
the inclusion homomorphisms $$H_k(X,X\setminus
K_1)\to H_k(X,X\setminus K_2)\to H_k(X,X\setminus K_3)$$ are trivial for all $k\le n$ (the existence of such neighborhoods $K_2,K_3$ follows from Lemma~\ref{homtriv}).

 These
trivial homomorphisms induce trivial homomorphisms
$$e_{2,1}:\mathrm{Ext}(H_{n-1}(X,X\setminus K_2),G)\to
\mathrm{Ext}(H_{n-1}(X,X\setminus K_1),G)$$ and
$$h_{3,2}:\mathrm{Hom}(H_n(X,X\setminus
K_3),G)\to\mathrm{Hom}(H_n(X,X\setminus K_2),G).$$ Now consider the
commutative diagram { $$
\hskip-3pt\xymatrix{
\mathrm{Ext}(H_{n{-}1}(X,X\backslash K_3),G)\ar[d]\ar[r] &H^n(X,X\backslash 
K_3;G)\ar[d]^{i_{3,2}} \ar[r]&\mathrm{Hom}(H_n(X,X\backslash  K_3),G)\ar[d]^{h_{3,2}}\\
\mathrm{Ext}(H_{n{-}1}(X,X\backslash  K_2),G)\ar[r]\ar[d]_{e_{2,1}}& H^n(X,X\backslash 
K_2;G)\ar[r]\ar[d]^{i_{2,1}} &\mathrm{Hom}(H_n(X,X\backslash  K_2),G)\ar[d]\\
\mathrm{Ext}(H_{n{-}1}(X,X\backslash  K_1),G)\ar[r]& H^n(X,X\backslash 
  K_1;G)\ar[r]&\mathrm{Hom}(H_n(X,X\backslash  K_1),G)
}
$$} The exactness of rows of the diagram and the triviality of the
homomorphisms $e_{2,1}$ and $h_{3,2}$ imply the triviality of the
homomorphism $i_{3,1}=i_{2,1}\circ i_{3,2}:H^n(X,X\setminus
K_3;G)\to H^n(X,X\setminus K_1;G)$. The local compactness of $X$
allows us to find an open $\sigma$-compact subset $W\subset X$
containing the compact set $K_1$. Since the sets $K_i$,
$i\in\{1,3\}$, are of type $G_\delta$ in $X$, the spaces
$W\setminus K_i$ are $\sigma$-compact and thus paracompact.

 The Excision Axiom for singular cohomology (see \cite[V.\S 4]{Spa}) implies that $H^n(X,X\setminus K_i;G)=H^n(W,W\setminus K_i;G)$ for $i\in\{1,3\}$. This observation implies that the inclusion homomorphism $i_{3,1}: H^n(W,W\setminus K_3;G)\to H^n(W,W\setminus K_1;G)$ is trivial. Since $W$, $W\setminus K_i$, $i\in\{1,3\}$, are paracompact $\lc[n]$-spaces, the singular cohomology group $H^n(W,W\setminus K_i;G)$ coincides with the \v Cech cohomology group $\check H^n(W,W\setminus K_i;G)$.
Consequently, the inclusion homomorphism $\check H^n(W,W\setminus K_3;G)\to \check H^n(W,W\setminus K_1;G)$ is trivial. By Excision Axiom for \v Cech cohomology, $\check H^n(W,W\setminus K_i;G)=\check H^n(X,X\setminus K_i;G)$ for $i\in\{1,3\}$. Consequently, the inclusion homomorphism $\check H^n(X,X\setminus K_3;G)\to \check H^n(X,X\setminus K_1;G)$ is trivial and so is the inclusion homomorphism $\check H^n(X,X\setminus \bar V;G)\to \check H^n(X,X\setminus \bar U;G)$, where $V$ is the interior of $K_3$. But this contradicts the choice of the neighborhood $U$.
\end{proof}

Theorem~\ref{cohdim} will help us to evaluate the extension dimension of a locally compact $\LC[n]$-space whose all points are homological $Z_n$-points.

\begin{theorem}\label{extdim} If a metrizable locally compact $\LC[n]$-space $X\in\Z^{\IZ}_n$ has $\edim(X)\le L$ for some CW-complex $L$, then $\pi_k(L)=0$ for all $k\le n$.
\end{theorem}

\begin{proof} Separately we shall consider the cases of $n=0,1$ and $n\ge 2$.

0. If $n=0$, then it suffices to check that the CW-complex $L$ is connected. Since each point of the $\LC[0]$-space $X$ is a homological $Z_0$-point, $X$ contains no isolated point and thus $X$ contains an arc $J$ connecting two distinct points $a,b\in X$. Assuming that the complex $L$ is disconnected,
consider any map $f:\{a,b\}\to L$ sending the points $a,b$ to different
components of $L$. Because of the connectedness of $J\subset X$
the map $f$ does not extend to $X$, which contradicts $\edim X\le
L$.
\smallskip

1. For $n=1$ we should prove the simple-connectedness of $L$.
Theorem~\ref{ind} implies that $\dim(X)>1$.
Then there is a point $x\in X$ whose any neighborhood
$U\subset X$ has dimension $\dim U>1$. Since $X$ is an $\LC[1]$-space, the
point $x$ has a closed neighborhood $N$ such that any map $f:\partial I^2\to N$ is null-homotopic in $X$. Moreover, we can assume that $N$ is a Peano continuum. Since
$\dim N>1$, the continuum $N$ is not a dendrite and consequently,
contains a simple closed curve $S\subset N$. Assuming that the
CW-complex $L$ is not simply-connected, we can find a map $f:S\to
L$ that is not homotopic to a constant map. Then the map $f$ cannot be extended over $X$ since the identity map of $S$ is null-homotopic in $X$. But this contradicts $\edim X\le L$.
\smallskip

2. Finally, we consider the case of $n\ge 2$. Suppose that $\pi_k(L)\ne 0$ for
some $k\le n$. We can assume that $k$ is the smallest number with
$\pi_k(L)\ne 0$. The simple connectedness of $L$ implies that
$k>1$. Applying Hurewicz Isomorphism Theorem, we conclude that
$H_k(L)=\pi_k(L)\ne 0$. Since $\edim X\le L$, we may apply
Theorem 7.14 of \cite{Dydak} to conclude that $\dim_{H_n(L)}X\le n$. But this
contradicts Theorem~\ref{cohdim}.
\end{proof}

 \section{Dimension of spaces whose all points are
$Z_\infty$-points}\label{dimension}

In this section we study the dimension properties of spaces whose all points are homological $Z_\infty$. We shall show that locally compact ANR's
with this property are infinite-dimensional in a rather strong
sense: they cannot be $C$-spaces and have infinite cohomological
dimension.

We recall that a topological space $X$ is defined to be a {\em
$C$-space} if for any sequence $\{\V_n\colon n\in\omega\}$ of open
covers of $X$ there exists a sequence $\{\U_n\colon n\in\omega\}$
of disjoint families of open sets in $X$ such that each $\U_n$
refines $\V_n$ and $\bigcup\{\U_n\colon n\in\omega\}$ is a cover
of $X$. By  Theorem 6.3.8 \cite{En} each metrizable countable-dimensional space is a $C$-space.

\begin{theorem}\label{infdim} If  $X\in\Z^{\IZ}_\infty$ is a locally compact metrizable  $\LC[0]$-space, then
\begin{enumerate}
\item $X$ fails to be $\trt$-dimensional and is not countable-dimensional;
\item if $X$ is an $\lc[\infty]$-space, then $\dim_GX=\infty$ for any abelian group $G$;
\item if $X$ is an $\LC[\infty]$-space, then $\edim X\not\le L$ for any non-contractible CW-complex $L$;
\item if $X$ is locally contractible, then $X$ fails to be a $C$-space.
\end{enumerate}
\end{theorem}

\begin{proof}
1--3. The first three items follow from Theorems~\ref{trans},
\ref{cohdim}, and \ref{extdim}, respectively.
\smallskip

4. First we prove the fourth item under a stronger assumption that
all points of $X$ are homotopical $Z_\infty$-points. Assume that
$X$ is a $C$-space. By Gresham's Theorem \cite{Gre}, the $C$-space $X$, being locally contractible, is an ANR. Then the product $X\times Q$ is a Hilbert cube manifold according to the Edwards' ANR-theorem, see
\cite[44.1]{Chap}. The product $X\times Q$, being a $Q$-manifold,
contains an open subset $U\subset X\times Q$ whose closure
$\overline{U}$ is homeomorphic to the Hilbert cube $Q$ and whose
boundary $\partial U$ is a $Z_\infty$-set in $\overline{U}$. Now
consider the multivalued map $\Phi:X\to \bar U$ assigning to each
point $x\in X$ the set $\Phi(x)=\overline{U}\setminus (\{x\}\times Q)$. Our
goal is to show that the map $\Phi$ has a continuous selection
$s:X\to \overline{U}$. Assuming that this is done, consider the
map $s\circ \pr:\overline{U}\to \overline{U}$, where
$\pr:\overline{U}\to X$ stands for the natural projection of
$\overline{U}\subset X\times Q$ onto the first factor. This map is
continuous and has no fixed point, which is a contradiction.

So, it remains to construct the continuous selection $s$ of the
multivalued map $\Phi$. The existence of such a selection will follow
from Uspenskij's Selection Theorem \cite{Us} as soon as we shall verify that
\begin{itemize}
\item for each $x\in X$ the complement
$\overline{U}\setminus \Phi(x)=\overline{U}\cap(\{x\}\times Q)$ is a
$Z_\infty$-set in $\overline{U}$;
\item for any compact set $K\subset\bar U$ the set $\{x\in
X:\Phi(x)\supset K\}$ is open in $K$.
\end{itemize}

The first condition holds since each point of $X$ is a
$Z_\infty$-point in $X$ and the boundary $\partial U$ is a
$Z_\infty$-set in $\overline{U}$. The second condition holds
because $\{x\in X:F(x)\supset K\}=X\setminus \pr(K)$.

This completes the proof of the special case when all points of
$X$ are $Z_\infty$-points. Now assume that all points of $X$ are
merely homological $Z_\infty$-points in $X$. Then the points of
the product $X\times [0,1]$ are (homotopical) $Z_\infty$-points by
Corollary~\ref{prodcor}. Now the preceding discussion implies that
$X\times [0,1]$ fails to be a $C$-space. Taking into account that
the product of a metrizable $C$-spaces with the interval is a
$C$-space \cite[Theorem 2.2.3]{AG}, we conclude that $X$ is not a
$C$-space.
\end{proof}

\begin{remark} The last item of Theorem~\ref{infdim} is true in a bit stronger form:
each locally compact locally contractible space $X\in\cup_G\Z^G_\infty$  fails to be a $C$-space, see \cite{BC}. However, the proof of this stronger result requires non-elementary tools like homological version of Uspenskij's Selection Theorem combined with a homological version of the Brouwer Fixed Point Theorem. On the other hand, this stronger result would follow from Theorem~\ref{trans} if each compact C-space were $\trt$-dimensional. However we are not sure that this is true.
\end{remark}

\begin{remark} In light of Theorem~\ref{infdim}, it is interesting to mention that a compact AR whose all points are $Z_\infty$-points need not be homeomorphic to the Hilbert cube. A suitable counterexample was constructed in \cite[9.3]{DW}.
\end{remark}

\begin{problem} Let $X$ be a compact absolute retract whose all
points are $Z_\infty$-points. Is $X$ strongly
infinite-dimensional? Is $X\times I$ homeomorphic to the
Hilbert cube?
\end{problem}

In this respect let us mention the following characterization of
Hilbert cube manifolds \cite{BR} which can be deduced from Theorem~\ref{indP1} and the homological characterization of Q-manifolds due to R.Daverman and J.Walsh \cite{DW}, see also \cite{BV}.

\begin{theorem} A locally compact ANR-space $X$ is a Hilbert cube manifold if and only if
\begin{enumerate}
\item $X$ has the Disjoint Disk Property;
\item each point of $X$ is a homological $Z_\infty$-point;
\item each map $f:K\to X$ of a compact polyhedron can be approximated by a map with $\trt$-dimensional image.
\end{enumerate}
\end{theorem}

We recall that a space $X$ has the {\em Disjoint Disk Property} if any two maps $f,g:I^2\to X$ can be approximated by maps with disjoint images.

Theorem~\ref{trans} implies that each closed $\trt$-dimensional subspace of the Hilbert cube $Q$ is a homological $Z_\infty$-set in $Q$. By Proposition 4.7 of \cite{ACP},  each compact $\trt$-dimensional space is a $C$-space.

\begin{question} Is a closed subset $A\subset Q$ a homological $Z_\infty$-set in $Q$ if $A$ is weakly infinite-dimensional or a $C$-space? \end{question}

%\newpage

\end{document}